\documentclass{article}

\usepackage[preprint]{neurips_2026}

\usepackage[utf8]{inputenc} 
\usepackage[T1]{fontenc}    
\usepackage{hyperref}       
\usepackage{url}            
\usepackage{booktabs}       
\usepackage{amsfonts}       
\usepackage{nicefrac}       
\usepackage{microtype}      
\usepackage{xcolor}         
\usepackage{array}
\usepackage{tabularx}
\usepackage{longtable}
\usepackage{placeins}
\usepackage{graphicx}
\usepackage{multirow}
\usepackage{siunitx}
\usepackage{caption}
\usepackage{subcaption}
\usepackage{amsmath,amssymb,amsthm,mathtools}
\usepackage{comment}
\newcommand{\NN}{\mathbb N}
\newcommand{\Nzero}{\mathbb N_0}
\newcommand{\PP}{\mathcal P}

\newcommand{\ind}[1]{\mathbf{1}_{\{#1\}}}

\newcommand{\EE}{\mathbb E}

\newcommand{\KL}{\mathrm{KL}}

\newcommand{\epower}{\widehat{\mathcal{P}}_M}

\newtheorem{theorem}{Theorem}[section]
\newtheorem{proposition}[theorem]{Proposition}
\newtheorem{lemma}[theorem]{Lemma}
\newtheorem{corollary}[theorem]{Corollary}
\theoremstyle{definition}
\newtheorem{definition}[theorem]{Definition}

\newtheorem{example}[theorem]{Example}

\title{Optimal sequential tests yield log-optimal e-processes}

\author{%
  Ashwin Ram\\
  Carnegie Mellon University\\
  Pittsburgh, PA 15213 \\
  \texttt{aram2@andrew.cmu.edu} \\
   \And
   Aaditya Ramdas\\
   Carnegie Mellon University \\
   Pittsburgh, PA 15213 \\
   \texttt{aramdas@cmu.edu} \\
}

\begin{document}

\maketitle

\begin{abstract}
    It has been recently shown that e-processes are sufficient for sequential testing in the following sense: every level-$\alpha$ sequential test can be obtained by thresholding an e-process at $1/\alpha$. However, in the above result, neither does the test have to be asymptotically optimal (in terms of stopping times) nor does the e-process have to be  asymptotically log-optimal. It has separately been shown that asymptotically log-optimal e-processes yield asymptotically optimal sequential tests. In this paper, we prove the converse, arguably completing the story: it is possible to aggregate asymptotically optimal sequential tests into asymptotically log-optimal e-processes. This is accomplished by using a new class of \emph{WAIT} e-processes: those that are \emph{Weighted Aggregates of Indicators of stopping Times} that begin at zero, are nondecreasing and increase to infinity under the alternative at the optimal rate.
    Importantly, the paper discusses several nuances in the varied definitions of asymptotic (log-)optimality. 
\end{abstract}

\section{Introduction}
Sequential tests are important because they provide the flexibility needed to adaptively stop a stream of data and reject the null hypothesis while uniformly controlling the Type-$I$ error at some specified rate $\alpha$. In its simplest form, a level-$\alpha$ sequential test is a stopping time $\tau_\alpha$ that satisfies $\sup_{P\in\PP}P(\tau_{\alpha}<\infty)\le \alpha$ for some null class $\PP$. The relationship between such tests and $e$-processes is an important question in the literature. Formally, an $e$-value is a nonnegative random variable with null expectation at most one, while an $e$-process is the sequential analogue and valid at any stopping time, for all $P\in\PP$. 

To this end, two key facts motivate our paper. First, it is known that thresholding any $\PP$-$e$-process $E=(E_t)$ at $1/\alpha$ will yield a level-$\alpha$ sequential test, thanks to Ville's inequality \citep{Ville1939}. Second, this representation is complete in the sense that any level-$\alpha$ sequential test can be recovered by thresholding some $e$-process at $1/\alpha$ \citep{RamdasWang2025}. However, this statement is still incomplete. Suppose the representing $e$-process is a two-valued indicator process. Then it will not necessarily have meaningful logarithmic growth under alternatives. In other words, on one hand, we have completeness of $e$-processes for validity. However, the question remains as to whether asymptotically optimal tests can be converted into asymptotically log-optimal evidence processes. 

In this paper, we therefore prove such a converse. Formally, suppose we are given only a level-indexed family of tests $(\tau_{\alpha_k})_{k\ge 1}$ with no pre-existing $e$-process. We will write $b_k:=\log(1/\alpha_k)$ and by definition note that $\sup_{P\in\PP}P(\tau_{\alpha_k}<\infty)\le \alpha_k$. Then for nonnegative weights $w_k$ that satisfy $\sum_k w_k\alpha_k\le 1$, we will define the weighted indicators of stopping times as
\[
M_t:=\sum_{k=1}^{\infty} w_k\ind{\tau_{\alpha_k}\le t}.
\]
We will refer to such processes as WAIT $e$-processes. These start at zero, are nondecreasing in time, and are valid $\PP$-$e$-processes (as we will prove). Importantly, their growth under an alternative $Q$ is governed by the deterministic cumulative weight profile $W(x):=\sum_{b_k\le x}w_k$. Suppose these tests are first-order optimal under $Q$ as in almost surely,
\[
\frac{\tau_{\alpha_k}}{b_k}\longrightarrow \frac{1}{I(Q)}.
\]
And further suppose that the profile satisfies
\[
\frac{\log W(x)}{x}\longrightarrow \rho.
\]
Then, under the alternative $Q$, we will have with probability one that
\[
\frac{\log M_t}{t}\longrightarrow \rho I(Q).
\]
Note that the constraint $\sum_{k}w_k\alpha_k\le 1$ implies that $W(x)\le e^x$. Thus, $\rho\le 1$. As a result, the WAIT aggregation is fully log-optimal whenever the levels and weights have the full profile exponent $\rho =1$. To this end, we will present a number of different level schedules where such a rate can be achieved. 

\subsection*{Related work}
The origins of such sequential tests go back to the most fundamental setting of a singleton null versus a singleton alternative. In this setting, it has been proven that likelihood-ratio thresholding is optimal in the sense of minimizing the expected sample size among all tests with the same error probabilities \citep{Wald1945, WaldWolfowitz1948}. A number of subsequent works then developed asymptotic and Bayes and minimax versions of this very principle, including a generalization of sequential probability ratio tests, asymptotically optimal adaptive sampling rules and sequential inference \citep{KeiferWeiss, chernoff1959, KiefferSacks}. These ideas were later refined for composite and multi-parameter problems: work was done in nearly optimal sequential tests for sets with a finite amount of parameters \citep{Lorden1977}, nearly optimal stopping rules and composite sequential tests \citep{lai1973, Lai1988}, systematized likelihood and boundary approximations \citep{Siegmund1986}. A number of further works then established first (and higher) order asymptotic optimality of methods like the mixture and generalized likelihood-ratio, going beyond independent observations to more generalized processes \citep{Dragalin1999, Fellouris2013}.

Separately, a number of works in the literature treat sequential testing as a testing-by-betting problem, with the linchpin of this viewpoint thanks to both Ville's work on non-negative martingales and Doob's optional stopping results \citep{Ville1939, doob1953stochastic}. In particular, Kelly and Breiman proved that for repeated favorable games, maximizing the logarithmic capital growth is the unique criterion for optimality in the limit \citep{kellyinterpretation, breiman1961optimal}. This growth-rate perspective was consequently developed in universal portfolio theory wherein mixtures lead to wealth processes whose logarithms will achieve the best strategy up to sublinear error \citep{cover1987log, cover1996}. Crucially, these ideas are the genesis for anytime-valid inference today. That is, a statistician will build a nonnegative capital process with controlled null expectation (thanks to optional stopping) and when the capital crosses $1/\alpha$, one rejects the null.

These two viewpoints are now linked with the work being done in both $e$-values and $e$-processes. A number of works have made fundamental contributions here: among others, calibration and multiple testing theory for $e$-values \citep{VovkWang2021}, developments in game-theoretic probability \citep{Shafer2021, RamdasGrunwaldVovkShafer2023}, exploiting nonnegative supermartingales to construct arbitrary confidence sequences (e.g.\ for all sample sizes, zero-width confidence intervals) \citep{HowardEtAl2021}, always-valid inference and more broadly safe anytime-valid inference \citep{johari2022, RamdasGrunwaldVovkShafer2023}. At the same time, there has been work in universal inference for obtaining finite-sample constructions in composite cases. Namely, split likelihood-ratio $e$-values \citep{wasserman2020universal}, development of growth-rate optimal $e$-variables in composite settings and then subsequent assumption-free strong duality for the growth rate under point alternatives \citep{GrunwaldDeHeideKoolen2024, LarssonRamdasRuf2025}. From these, a number of wonderful applications have emerged. For example, safe $e$-processes for exchangeability testing \citep{ramdas2022testing}, bounded-mean and empirical-Bernstein methods \citep{WaudbySmithRamdas2024}, sequential nonparametric change detection through $e$-detectors \citep{shin2024edetectors} which are provably optimal \citep{ram2026asymptoticallyoptimalsequentialchange}, sequential independence testing \citep{podkopaev2023}, exact sequential $k$-sample testing \citep{turner2024}. However, none of these works have analyzed the problem of aggregating sequential tests to log-optimal processes, which is the main contribution of our work.

\subsection*{Contributions}
We summarize our main contributions as follows:
\begin{enumerate}
    \item We introduce the concept of WAIT $e$-processes and prove their validity for any arbitrary composite nulls in full generality.
    \item We prove an exact deterministic profile theorem. That is, if the tests are first order optimal, the resulting $e$-process will have log-growth rate $\rho I(Q)$, with $\rho$ determined entirely by the level and weight profile. We valid these results empirically as well.
    \item We identify several full-rate $\alpha_k$ schedules and weights. These include log-corrected unweighted schedules and weighted geometrically space schedules.
    \item We cleanly present the first unification of several notions of optimality. These include almost-sure, in-probability, $L^{1}$, expectation (with concentration), and threshold-time optimality. In turn, we give intuition on exactly what assumptions are needed to convert optimal sequential tests (in any sense) to log-optimal $e$-processes. And in doing so, we prove (and validate empirically) why a very common notion of optimality in the testing literature (expectation optimality) on its own is not sufficient to aggregate into an optimal $e$-process. 
\end{enumerate}

The rest of this paper is organized as follows. In Section~\ref{sec:warmup}, we formally define the level-indexed sequential tests, the $\PP$-$e$-process, and the WAIT aggregate and profile. Then, in Section~\ref{sec:cute}, we give our general WAIT aggregation theorem, the full-profile schedules that yield log-optimal $e$-processes, and we show that the threshold tests derived from the full-profile WAIT processes are first-order optimal. In Section~\ref{sec:cuter}, we present the unweighted profile theorem and a weighted improvement of a schedule that would not normally deliver log-optimality in the aggregate. We relate several types of optimality in Section~\ref{sec:cutest}, as well as what notions of test optimality are needed to obtain a log-optimal aggregated $e$-process. We conclude in Section~\ref{sec:cutestest}. In Appendix~\ref{sec:exp1} and Appendix~\ref{sec:exp2} we provide experiments that validate our theoretical results. Further, all proofs of the results can also be found in the appendix.

\section{Problem setup and WAIT aggregation}\label{sec:warmup}

Throughout, we work on a filtered measurable space $(\Omega, \mathcal F, (\mathcal F_t)_{t\in\mathbb N_0})$. We say that the composite null is a family $\mathcal P$ of probability measures on $(\Omega, \mathcal F)$. We define level-indexed sequential tests as follows.

\begin{definition}
Suppose that $(\alpha_k)_{k\ge 1}$ is some decreasing sequence in $(0,1)$ with $\alpha_k\downarrow 0$. We will write $b_k:=\log(1/\alpha_k)$. For each $k\in\mathbb N$, let $\tau_k:=\tau_{\alpha_k}$ be an $(\mathcal F_t)$-stopping rule that takes values in $\mathbb N_0\cup\{\infty\}$ such that
\begin{equation}\label{eq:V}
\sup_{P\in\mathcal P} P(\tau_k<\infty)\le \alpha_k.
\end{equation}
\end{definition}

Note that this construction only needs to have tests at each of the $(\alpha_k)$ selected levels. Starting with a full family $\{\tau_\alpha\}_{0<\alpha<1}$, one can choose some level schedule $(\alpha_k)$ and then take $\tau_k:=\tau_{\alpha_k}$. Let us now define what we mean by a strong asymptotic rate at the chosen levels.

\begin{definition}
Take some particular alternative probability measure $Q$ and some constant $I\in(0,\infty)$. We say that the chosen level tests have a pathwise rate of $I$ under $Q$ if
\begin{equation}\label{eq:AO}
\frac{\tau_k}{b_k}\longrightarrow \frac{1}{I},
\end{equation}
$Q$ almost surely as $k\to\infty$.    
\end{definition}

We now define an $e$-process. 

\begin{definition}
We say that a nonnegative adapted process $M=(M_t)_{t\in\mathbb N_0}$ is a $\mathcal P$-$e$-process if $\EE_P[M_\sigma]\le 1$ for every $P\in\mathcal P$ and every finite $(\mathcal F_t)$-stopping time $\sigma$.
\end{definition}

Note that strict positivity here will come at no cost asymptotically. In particular, strict positivity is optional here, as several of our processes throughout this paper will start at $0$. However, to obtain a strictly positive $e$-process, one may choose $\eta\in(0,1)$ and define $M_t^{(\eta)}:=\eta + (1-\eta)M_t$. Then, $M^{(\eta)}$ is again a $\mathcal P$-$e$-process because we have for every finite stopping rule $\tau$ that
\[
\EE_P[M_\sigma^{(\eta)}]=\eta+(1-\eta)\EE_P[M_\sigma]\le \eta + (1-\eta) \cdot 1 = 1.
\]
And moreover, whenever $M_t\to\infty$, along a given path it holds that
\[
\frac{1}{t}\log M_t^{\eta}-\frac{1}{t}\log M_t \longrightarrow 0.
\]
Therefore, strict positivity can be obtained for free in the sense that it will not come at any asymptotic cost. Now, take some particular nonnegative weights $(w_k)_{k\ge 1}$ (not all zero) which satisfy
\begin{equation}\label{eq:W}
\sum_{k=1}^{\infty}w_k\alpha_k\le 1.    
\end{equation}

For some $t\in\mathbb N_0$, let us define the weighted aggregate as
\begin{equation}\label{eq:aggregate}
M_t:=\sum_{k=1}^\infty w_k \ind{\tau_k\le t}.
\end{equation}

For $x\ge 0$, let us define the cumulative weight profile as
\begin{equation}\label{eq:Wprofile}
W(x):=\sum_{k:b_k\le x}w_k.    
\end{equation}

Since $\alpha_k\downarrow 0$, we have that $b_k\uparrow \infty$. As such, for each finite $x$, there are only finitely many indices $k$ where $b_k\le x$. Therefore, $W(x)$ is always a finite sum. With this in mind, we now formalize that the weighted indicator aggregation is always an $e$-process.

\begin{proposition}\label{prop:eprocess-validity}
If \eqref{eq:V} and \eqref{eq:W} hold, then the process $M$ in \eqref{eq:aggregate} is a nondecreasing $\mathcal P$-$e$-process.
\end{proposition}

We will now give a proposition showing that consistency will transfer from the test family.

\begin{proposition}\label{prop:consistency-transfer}
Assume that \eqref{eq:W} holds and suppose that for some index $k_0\in\mathbb N$,
\begin{equation}\label{eq:C}
Q(\tau_k<\infty \text{ for every }k\ge k_0)=1.  
\end{equation}
Assume also that $\sum_{k=k_0}^\infty w_k=\infty$. Then, $M_t\to\infty$ almost surely under $Q$ as $t\to\infty$.
\end{proposition}

Note that the cumulative weight profile can never grow faster than $e^x$. We will formalize that as follows.

\begin{lemma}\label{lem:profile-upper-bound}
Assume that \eqref{eq:W} holds. Then for every $x\ge 0$, $W(x)\le e^x$. Consequently, whenever the limit $\rho:=\lim_{x\to\infty}\frac{\log W(x)}{x}$ exists, $\rho\le 1$.
\end{lemma}

\section{The WAIT Aggregation Theorem}\label{sec:cute}

We now prove that the asymptotic logarithmic growth is determined by $W$.

\begin{theorem}\label{thm:main}
Assume that \eqref{eq:V}, \eqref{eq:W}, and \eqref{eq:AO} hold. Suppose that $M$ is the aggregate (as in \eqref{eq:aggregate}) and let $W$ be the profile (as in \eqref{eq:Wprofile}). Suppose in addition $W(x)\to\infty$ and for some $\rho\in[0,1]$ we have that
    \begin{equation}\label{eq:R}
    \frac{\log W(x)}{x}\longrightarrow \rho.   
    \end{equation}
    Then, almost surely under $Q$ as $t\to\infty$,
    \begin{equation}\label{eq:G}
    \boxed{
    \frac{1}{t}\log M_t\longrightarrow \rho I.  
    }
    \end{equation}
\end{theorem}

Intuitively, Theorem~\ref{thm:main} shows that once the tests reach an asymptotic speed of $b_k/I$, the aggregate will have a log growth rate that is simply the product of $I$ with the exponential growth rate of $W(x)$. Therefore, the entire problem becomes about building a valid schedule with a cumulative weight profile $W(x)$ that is as large as possible. However, recall that thanks to Lemma~\ref{lem:profile-upper-bound}, it is not possible to exceed $e^x$. Therefore, the maximum possible exponent in this indicator aggregation class is $1$.

\section{Profiles and Schedules}\label{sec:cuter}
Let us now specialize to the unweighted case where $w_k\equiv 1$. Then
\begin{equation}\label{eq:count}
B_t:=\sum_{k=1}^\infty \ind{\tau_k\le t}    
\end{equation}
is the counting process with the chosen level schedule. In this case, the cumulative weight profile will simply be the counting function
\begin{equation}\label{eq:counting-profile}
N(x):=\#\{k\in\NN: b_k\le x\}.    
\end{equation}

Now, because $b_k\uparrow \infty$, this equals the largest index $k$ for which $b_k\le x$. We will adopt the convention where $N(x)=0$ when no index exists. All of this therefore motivates the following corollary.

\begin{corollary}\label{cor:unweighted}
Assume that $\sum_{k=1}^{\infty}\alpha_k\le 1$ and suppose both \eqref{eq:V} and \eqref{eq:AO} hold. Let $B$ be defined by \eqref{eq:count}. And let $N$ be given by \eqref{eq:counting-profile}. Suppose that $N(x)\to\infty$ and that for some $\rho\in[0,1]$,
\[
\frac{\log N(x)}{x}\longrightarrow \rho.
\]
Then it follows that almost surely under $Q$ as $t\to\infty$,
\[
\frac{1}{t}\log B_t\longrightarrow \rho I.
\]
\end{corollary}

We will now present a lemma to invert the asymptotic relationship between the level scale $b_k$ and the counting profile $N(x)$. This lemma will be used to convert level schedules into concrete growth rates.

\begin{lemma}\label{lem:inversion}
Suppose that $(b_k)_{k\ge 1}$ is increasing with $b_k\to\infty$. And let $N(x):=\#\{k:b_k\le x\}$. Then the following two statements hold.
\begin{enumerate}
    \item If for some $c\in(0,\infty)$ that $b_k/(\log k)\to c$, then it follows that
    \[
    \frac{\log N(x)}{x}\longrightarrow \frac{1}{c}.
    \]
    \item If $b_k/(\log k)\to\infty$, then it follows that
    \[
    \frac{\log N(x)}{x}\longrightarrow 0.
    \]
\end{enumerate}
\end{lemma}

We now present a quick corollary of rate classification for unweighted schedules.

\begin{corollary}\label{cor:classification}
Assume that $\sum_{k=1}^{\infty}\alpha_k\le 1$ and that \eqref{eq:V} and \eqref{eq:AO} both hold. Then the following statements hold.
\begin{enumerate}
    \item Suppose for some $c\in(0,\infty)$, $b_k=c\log k + o(\log k)$. Then almost surely under $Q$, we have
    \[
    \frac{1}{t}\log B_t\to \frac{I}{c}.
    \]
    \item If $b_k/(\log k)\longrightarrow \infty$, then almost surely under $Q$,
    \[
    \frac{1}{t}\log B_t\longrightarrow 0.
    \]
    \item Suppose that $b_k=\log k + o(\log k)$. Equivalently, suppose $\log(1/\alpha_k)=\log k + o(\log k)$. Then, the unweighted counting construction will achieve the full rate $I$. That is, we have almost surely under $Q$ that
    \[
    \frac{1}{t}\log B_t\longrightarrow I.
    \]
\end{enumerate}
\end{corollary}

To construct these full-rate unweighted schedules, one can intuit the process as follows. Write $\alpha_k=e^{-b_k}$. Then the sufficient condition of $b_k=\log k + o(\log k)$ from the third statement of Corollary~\ref{cor:classification} is equivalent to taking $\alpha_k = e^{-h(k)}/k$ with $h(k)=o(\log k)$ up to some normalization constant. For an unweighted indicator aggregation to be valid, it suffices to make $\sum_k \alpha_k<\infty$ because one can rescale the levels by some constant to enforce the overall $\sum_k \alpha_k\le 1$. 

The question therefore becomes how to choose this $h(k)$ such that $h(k)\to\infty$ fast enough to keep $h(k)=o(\log k)$, while also being fast enough to make $\sum_{k=1}^\infty \frac{e^{-h(k)}}{k}<\infty$. Certainly, the examples $h(k)=p\log\log k$ with $p>1$ and $h(k)=\log\log k + 2\log\log \log k$ both satisfy these requirements.

If we further assume that $(\alpha_k)$ is strictly decreasing (which means that $(b_k)$ is strictly increasing) and the limit $\lim_{x\to\infty}\log N(x)/x$ exists, then one may wonder what happens in this case also. In the same setting of Corollary~\ref{cor:unweighted} where the full unweighted rate $I$ is achieved, this will be equivalent to $\log N(x)/x\to 1$. Therefore, when evaluating at $x=b_k$, since $N(b_k)=k$, $b_k=\log k + o(\log k)$ holds. With this intuition in mind, we now provide a full framework for achieving the full-rate in the unweighted case. 

\begin{proposition}\label{prop:general-h}
Assume that \eqref{eq:V} and \eqref{eq:W} hold. Take some particular index $k_0\ge 3$. Let $h:\{k_0, k_0+1, \dots\}\to (0,\infty)$ be nondecreasing and satisfy $h(k)=o(\log k)$ and $\sum_{k=k_0}^\infty e^{-h(k)}/k <\infty$. Define for $k\ge k_0$, $\alpha_k:=c_h\frac{e^{-h(k)}}{k}$, where
\[
c_h:=\left(\sum_{k=k_0}^\infty \frac{e^{-h(k)}}{k}\right)^{-1}.
\]
Also, define
\[
B_t^{(h)}:=\sum_{k=k_0}^\infty \ind{\tau_{\alpha_k}\le t}.
\]
And assume further that under $Q$ almost surely as $k\to\infty$,
\[
\frac{\tau_{\alpha_k}}{\log(1/\alpha_k)}\longrightarrow \frac{1}{I}.
\]
Then it follows that $B^{(h)}$ is an $e$-process. And, almost surely under $Q$, we have
\[
\boxed{
\frac{1}{t}\log B_t^{(h)} \longrightarrow I.
}
\]
\end{proposition}

In this unweighted case, we summarize in the below table several schedules that can be chosen, whether a valid $e$-process can be extracted, and what can be said about log optimality. All the entries will be proven in the appendix. Note that all schedules are valid up to some normalization constant.

\begin{center}
\begin{tabular}{@{}llll@{}}
\toprule
Schedule $\alpha_k$ & Valid? & $\log(1/\alpha_k)$ scale & Resulting log-rate\\ \midrule
$2^{-k}$ & yes & $k\log 2$ & $0$\\
$k^{-(1+\varepsilon)}$ & yes for $\varepsilon>0$ & $(1+\varepsilon)\log k$ & $I/(1+\varepsilon)$\\
$1/[k(\log k)^p]$ & yes for $p>1$ & $\log k + p\log\log k$ & $I$\\
$1/[k\log k(\log\log k)^2]$ & yes & $\log k + \log\log k + 2\log\log\log k$ & $I$\\
\bottomrule
\end{tabular}    
\end{center}

\subsection*{Weighted Improvement}
It has been shown that we may recover the full rate $I$ with an unweighted count, assuming that we pick levels that decay like the product of $1/k$ with a subpolynomial correction factor. However, a more straightforward idea is to keep the levels $2^{-k}$ and only change the weights. We formalize that intuition as follows: a weighted construction that achieves the full rate.

\begin{proposition}\label{prop:weighted-dyadic}
Let us define for $t\in\mathbb N_0$ the aggregated process as
\(
M_t^{\mathrm{dy}}:=\frac{6}{\pi^2} \sum_{k=1}^\infty \frac{2^k}{k^2}\,\ind{\tau_{2^{-k}}\le t}.
\)
Assume in addition that almost surely under $Q$,
\[
\frac{\tau_{2^{-k}}}{k\log 2}\longrightarrow \frac{1}{I}.
\]
Then, $M^{\mathrm{dy}}$ is a $\mathcal P$-$e$-process. And almost surely under $Q$ we have that
\[
\frac{1}{t}\log M_t^{\mathrm{dy}}\longrightarrow I.
\]
\end{proposition}

This weighted construction is the easiest way to fix the construction in \cite{ruf2023composite}, which in the unweighted case does not achieve log-optimality (as shown with the resulting log-rate of $0$ in the above table). In this case, one does not change the schedule, but the capital that is assigned to each level.

\section{Weaker Notions of Optimality}\label{sec:cutest}

Unfortunately, the almost sure rate of the tests is a fairly strong assumption, albeit it delivers optimality of the aggregate if $I$ is optimal. One may wonder if it is feasible to weaken this assumption to other notions. That will be the focus of this section, where we prove what weaker notions of test-time optimality are sufficient to achieving optimality of the aggregate. 

Throughout this section, we will consider some alternative $Q$ and some information rate benchmark $I=I(Q)\in(0,\infty)$ and levels $b_k=\log(1/\alpha_k)\to\infty$. We now relate various types of test-time optimality to the almost-sure one used throughout the paper.

\begin{definition}\label{def:test-optimality}
We say that a level-indexed family of sequential tests $(\tau_k)$ is one of four different things, defined as follows.
\begin{enumerate}
    \item $\tau_k$ has a \underline{$Q$-almost-sure rate} of $I$ if
    \[
    \frac{\tau_k}{b_k}\to \frac{1}{I} \text{ almost surely under $Q$.}
    \]
    \item $\tau_k$ has a  \underline{$Q$-probability rate} of $I$ if
    \[
    \frac{\tau_k}{b_k}\to \frac{1}{I}\text{ in probability under $Q$.}
    \]
    \item $\tau_k$ has an \underline{$L^{1}(Q)$-rate} of $I$ if
    \[
    \frac{\EE_Q\left[\left|\tau_k-\frac{b_k}{I}\right|\right]}{b_k}\to 0.
    \]
    \item $\tau_k$ has an \underline{expectation-rate} of $I$ if $\EE_Q[\tau_k]<\infty$ eventually and
    \[
    \frac{\EE_Q[\tau_k]}{b_k}\to\frac{1}{I}.
    \]
\end{enumerate}
\end{definition}
The implications between these are straightforward. Although $I$ can be any rate, suppose $I$ is optimal. Then, pathwise optimality clearly implies optimality in $Q$-probability. In addition $L^{1}$ optimality implies both $Q$-probability optimality (thanks to Markov's inequality, as shown later) and expectation optimality. Additionally, Definition~\ref{def:test-optimality} allows one to formally understand which notions weaker (than almost sure) can deliver optimality. To this end, suppose that $E=(E_t)$ is a $\PP$-$e$-process. Then, $E$ has an almost sure log-growth rate $L$ under $Q$ if $\log(E_t)/t\to L$ with probability one under $Q$. Similarly, it will have log-growth rate $L$ in $Q$-probability if $\log(E_t)/t\to L$ in $Q$-probability. 

Now, if $I(Q)$ is some benchmark for $Q$, then $E$ will be called almost surely log-optimal under $Q$ if it achieves an almost sure log-growth rate $I(Q)$. For some $\rho\in[0,1]$, if the rate is $\rho I(Q)$, then we call $E$ $\rho$-log-efficient. 

Having said this, we will define \underline{threshold-time} optimality as follows. 

\begin{definition}\label{def:threshold-time-optimality}
For some nonnegative process $E$, let us define its threshold rejection time as
\[
T_{\alpha}(E):=\inf\{t\in\mathbb N_0: E_t\ge 1/\alpha\}.
\]
If $E$ is a $\PP$-$e$-process, then it follows that $T_{\alpha}(E)$ is a level $\alpha$-sequential test. We say that the threshold tests achieve an \underline{almost sure} first-order rate $I$ if with probability one under $Q$ as $\alpha\downarrow 0$,
\[
\frac{T_{\alpha}(E)}{\log(1/\alpha)}\to \frac1I.
\]
The in-probability version is defined analogously as this.
\end{definition}

With this defined, note that log-growth will imply threshold-time optimality. We will formalize that as follows.

\begin{proposition}\label{prop:log-growth-to-threshold}
Suppose that $E=(E_t)$ is nondecreasing and finite valued. Set $T_{\alpha}(E):=\inf\{t:E_t\ge 1/\alpha\}$. Suppose for some $L\in(0,\infty)$ we have that almost surely under $Q$,
\[
\frac{\log E_t}{t}\to L.
\]
Then it follows that almost surely under $Q$ as $\alpha\downarrow 0$,
\[
\frac{T_{\alpha}(E)}{\log(1/\alpha)}\to \frac1L.
\]
If $\log(E_t)/t\to L$ in probability under $Q$, then this same threshold-time convergence will hold in $Q$-probability instead.
\end{proposition}

Proposition~\ref{prop:log-growth-to-threshold} allows us to conclude that under our assumptions in Theorem~\ref{thm:main}, the threshold tests generated by our WAIT $e$-process will satisfy almost surely under $Q$ that
\[
\frac{T_{\alpha}(M)}{\log(1/\alpha)}\to \frac{1}{\rho I}.
\]

Here, $T_{\alpha}(M):=\inf\{t:M_t\ge 1/\alpha\}$ is the threshold test and $M_t$ is the aggregate (the WAIT $e$-process). The takeaway is that the full-profile WAIT aggregates (with $\rho=1$) will give rise to first-order optimal tests at rate $I$.

Unfortunately, convergence in $Q$-probability will not be enough for the simultaneous control that we were able to obtain in the almost sure optimality case (and essential for the aggregation to work). However, if the tests are nested then this in probability convergence is enough. To elaborate, we will first define a nested level family.

\begin{definition}\label{def:nested}
We say that the particular tests we choose are nested if $j\le k$ implies that $\tau_j\le\tau_k$ pathwise on $\Omega$. Equivalently, for each particular $t$, the events $\{\tau_k\le t\}$ are decreasing in $k$.    
\end{definition}

One can immediately see that thresholded families of the form $\inf\{t\in \NN_0: S_t\ge \log(1/\alpha)\}$ are automatically nested in $\alpha$. Remarkably, for nested families, in probability optimality for sequential tests is enough to deliver optimality of the aggregate, an idea formalized as follows.

\begin{theorem}\label{thm:nested-inprob}
Assume that \eqref{eq:V}, \eqref{eq:W}, and the nested property (as in Definition~\ref{def:nested}) all hold. Assume further that $W(x)\to\infty$ and that \eqref{eq:R} holds with $\rho\in[0,1]$. Suppose now we have in $Q$-probability that
\[
\frac{\tau_k}{b_k}\longrightarrow \frac{1}{I}.
\]
Then, in $Q$-probability as $t\to\infty$ it follows that
\[
\frac{1}{t}\log M_t\longrightarrow \rho I.
\]
\end{theorem}

Even though in-probability is weaker certainly than our almost-sure requirement, in the sequential testing literature, expectation-optimality is most typical. Unfortunately, this alone will not be sufficient as pointed out in Example~\ref{ex:expectation-too-weak}. However, if expectation-optimality and concentration are assumed, then that does deliver log-optimality of the aggregate, formalized as follows.

\begin{corollary}\label{cor:expectation-plus-concentration}
Assume that \eqref{eq:V}, \eqref{eq:W}, and the nestedness property (as in Definition~\ref{def:nested}) all hold. Assume also that $W(x)\to\infty$ and that \eqref{eq:R} holds with $\rho\in[0,1]$. Also, suppose that one of the following conditions hold.
\begin{enumerate}
    \item We have $L^1$ optimality:
    \[
    \frac{\EE_Q\left[\left|\tau_k-\frac{b_k}{I}\right|\right]}{b_k}\longrightarrow 0.
    \]
    \item Alternatively, we have:
    \[
    \frac{\EE_Q[\tau_k]}{b_k}\longrightarrow\frac{1}{I}\quad\text{and}\quad \frac{\operatorname{Var}_Q(\tau_k)}{b_k^2}\longrightarrow 0.
    \]
\end{enumerate}
Then it follows that in $Q$-probability as $t\to\infty$,
\[
\frac{1}{t}\log M_t\longrightarrow\rho I.
\]
\end{corollary}

Lastly, if we also have summable deviations, then that also gives us pathwise optimality. This is the strongest form of log-optimality for the aggregate.

\begin{corollary}\label{cor:summable-deviations}
Assume that for every rational $\varepsilon\in(0,1)$,
\[
\sum_{k=1}^\infty Q\left(\left|\frac{\tau_k}{b_k}-\frac{1}{I}\right|>\frac{\varepsilon}{I}\right)<\infty.
\]
Then, \eqref{eq:AO} holds. That is, the almost-sure growth that follows from \eqref{eq:G} of Theorem~\ref{thm:main} holds.
\end{corollary}

As remarked, it is important to conceptually see why on its own, the expectation optimality is too weak for the purposes of this work, as shown in Example~\ref{ex:expectation-too-weak} below. It formalizes why expectation optimality alone is not enough. In particular, $\EE_Q[\tau_k]/b_k\to 1/I$ does not imply in-probability optimality, almost sure optimality, or any deterministic log-growth of a WAIT aggregate. That is why, throughout this paper, any result that starts from mean optimality needs an additional assumption. This could be $L^{1}$ optimality, a normalized variance that vanishes, or any other assumption with such control over the aggregate.

\section{Conclusion}\label{sec:cutestest}
Prior literature has established that $e$-processes are sufficient for anytime-valid sequential testing in the sense that thresholding such a $\PP$-$e$-process at $1/\alpha$ gives a level-$\alpha$ sequential test. Conversely, thresholding an appropriate $e$-process allows recovery of any level-$\alpha$ sequential test. In spite of this, a major issue that remained is how to ensure that an asymptotically optimal family of sequential tests give an $e$-process that achieves asymptotically optimal logarithmic growth under any alternative $Q$. This issue is addressed by our paper. That is, we prove under an apt level-weight profile condition, an aggregation that preserves optimality is possible. In addition, we relate several notions of asymptotic optimality and clarify the conditions under which it is possible to achieve (almost-sure) optimal logarithmic-growth of the aggregated process. In doing so, we distinguish between validity, consistency, optimality at test-time, growth-rate optimality, and threshold-time optimality. To our knowledge, we are the first work to do so.

In spite of these novelties, there are numerous remaining open problems. For one, we identify the limit of $\log M_t/t$, but we do not give second order terms or nonasymptotic lower bounds on $M_t$ under the alternative. In addition, we assume that the input tests are optimal. Although in Appendix~\ref{sec:extremelycute}, we give sufficient conditions for obtaining optimal tests for some model, it is not obvious when such tests exist, as that is largely dependent on the choice of model. Thirdly, our results are log-optimal pointwise under each alternative $Q$. It would be interesting to extend these results to be simultaneously log-optimal over an entire class of alternatives $\mathcal Q$. With all this being said, our novel WAIT construction is an intuitive and model-independent way of converting optimal sequential tests to optimal sequential evidence processes. And in doing so, we clarify what assumptions are needed to go from different notions of test optimality to log-optimality of the resulting aggregated $e$-process. In doing so, we complete the narrative of the relationship between sequential tests and $e$-processes.

\bibliographystyle{unsrtnat}
\bibliography{references}

\appendix
\section{A counterexample to why expectation optimality alone is insufficient}
\begin{example}\label{ex:expectation-too-weak}
Consider any particular but arbitrarily chosen increasing sequence $b_k\uparrow\infty$ and any nonnegative weights $(w_k)$ that satisfy \eqref{eq:W}. Take a single null $\PP=\{P\}$. We will work on a filtered space with constant filtration. That is, for $t\in\NN_0$, $\mathcal F_t\equiv\sigma(Y,B_1,B_2,\dots)$. Here, $Y$ and $(B_k)$ are random variables that satisfy $Q(B_k=1)=1$ for all $k$, $P(B_k=1)=\alpha_k$ for all $k$, and 
\[
Q\left(Y=\frac{1}{2}\right)=Q\left(Y=\frac{3}{2}\right)=\frac12.
\]
Now, define
\[
\tau_k:=
\begin{cases}
\left\lceil \dfrac{Yb_k}{I}\right\rceil,&B_k=1,\\
\infty,&B_k=0.
\end{cases}
\]
Each $\tau_k$ is an $(\mathcal F_t)$-stopping rule and $P(\tau_k<\infty)=P(B_k=1)=\alpha_k$. Therefore the required size bound (i.e. \eqref{eq:V}) holds for the null family $\PP=\{P\}$. Under $Q$, $B_k=1$ almost surely for every $k$. Hence, $\tau_k=\lceil\frac{Yb_k}{I}\rceil$ almost surely under $Q$. Therefore it follows that almost surely under $Q$ we have
\[
\frac{\tau_k}{b_k}\longrightarrow \frac{Y}{I}.
\]
In addition,
\[
\frac{\EE_Q[\tau_k]}{b_k}=\frac{\EE_Q[Y]}{I}+O\left(\frac{1}{b_k}\right)\longrightarrow \frac{1}{I}.
\]
By definition, this family is expectation-optimal at rate $I$. However, note that $Y$ is nondegenerate. So $\tau_k/b_k$ does not converge in $Q$-probability to $1/I$. In any case, consider the corresponding aggregate
\[
M_t:=\sum_{k=1}^\infty w_k\ind{\tau_k\le t}.
\]
With this aggregate one has almost surely under $Q$ that
\[
M_t:=\sum_{k=1}^\infty w_k \ind{b_k\le It/Y}=W\left(\frac{It}{Y}\right).
\]
Consequently suppose that \eqref{eq:R} holds. Then it follows that with probability one under $Q$,
\[
\frac{1}{t}\log M_t\longrightarrow \frac{\rho I}{Y}.
\]
This is a random quantity. Therefore, expectation optimality is much much too weak to conclude \eqref{eq:G}, or any such deterministic time control.
\end{example}

\section{An Equivalent Version of the WAIT Aggregation Theorem}
Our almost-sure assumption is equivalent to tail-uniform optimality in $Q$-probability. Informally, it can be viewed as an alternate way to intuit our main results.

\begin{definition}\label{def:tail-prob-opt}
We will say that the level tests we choose are \underline{tail uniformly optimal} at rate $I$ in $Q$-probability if for every $\varepsilon\in(0,1)$ and every $\delta\in(0,1)$, there exists some $K=K(\varepsilon,\delta)$ such that
\[
Q\left(\frac{1-\varepsilon}{I} b_k\le\tau_k\le\frac{1+\varepsilon}{I}b_k\text{ for all } k\ge K \right)\ge 1-\delta.
\]
\end{definition}

The reason this above tail uniform optimality can be viewed as equivalent to almost sure convergence can be argued as follows. For the implication from almost sure convergence to tail uniform optimality, it is immediate by continuity from below. That is, for all $\varepsilon$, the events
\[
A_K^\varepsilon:=\left\{\frac{1-\varepsilon}{I}b_k\le \tau_k\le\frac{1+\varepsilon}{I}b_k\text{ for all }k\ge K \right\}
\]
will increase to an event of probability one. Conversely, if tail uniform optimality holds, then for all $\varepsilon$, we have that $Q\left(\bigcup_{K=1}^{\infty} A_{K}^{\varepsilon}\right)=1$. If we now intersect over all rationals $\varepsilon\in(0,1)$ we have that $\tau_k/b_k\to 1/I$ almost surely under $Q$.

We are now ready to present the theorem we have been building towards: an analogue of Theorem~\ref{thm:main}.

\begin{theorem}\label{thm:main-prob}
Assume that \eqref{eq:V}, \eqref{eq:W}, and Definition~\ref{def:tail-prob-opt} all hold. Suppose that $M$ is the aggregate (as in \eqref{eq:aggregate}). Further suppose that $W(x)\to \infty$ and that \eqref{eq:R} holds for some $\rho\in[0,1]$. Then it follows that in $Q$-probability as $t\to\infty$,
\[
\frac{1}{t}\log M_t\longrightarrow \rho I.
\]
\end{theorem}

Note that this Theorem~\ref{thm:main-prob} is the probability analogue of Theorem~\ref{thm:main}. In particular, it necessitates that we must have simultaneous control of all large enough levels on a high probability event. Although it concludes a weaker assumption, we will be using the techniques of this proof for our assumption-free in-probability proof.

\section{Omitted Proofs for $e$-process validity and WAIT comparison lemmas}
\begin{proof}[Proof of Proposition~\ref{prop:eprocess-validity}]
Take some $t\in\mathbb N_0$. For each $k$, the event $\{\tau_k\le t\}$ belongs to $\mathcal F_t$ since $\tau_k$ is a stopping time. Therefore, the random variable $w_k\ind{\tau_k\le t}$ is $\mathcal F_t$ measurable. For each $n\in\NN$, let us now define the partial sums
\[
M_t^{(n)}:=\sum_{k=1}^n w_k \ind{\tau_k\le t}.
\]
By definition, each $M_t^{(n)}$ will be $\mathcal F_t$ measurable. And as $n\to\infty$, $M_t^{(n)}\uparrow M_t$ pointwise by positivity of all summands. Hence, $M_t$ is $\mathcal F_t$ measurable, namely since it is the monotone limit of $\mathcal F_t$-measurable random variables. Now, if $s\le t$, then for every $k$, $\ind{\tau_k\le s}\le \ind{\tau_k\le t}$. If we multiply by $w_k\ge 0$ and then sum over all $k$, this gives $M_s\le M_t$, entailing that the process is non-decreasing. Now, let us have $\sigma$ be any finite stopping rule and consider some $P\in\PP$. Since all summands are nonnegative, thanks to Tonelli's theorem we have
\[
\EE_P[M_\sigma]=\EE_P\left[\sum_{k=1}^{\infty}w_k\ind{\tau_k\le \sigma}\right]=\sum_{k=1}^{\infty}w_k P(\tau_k\le\sigma).
\]
Since $\sigma$ is almost surely finite, for each $k$ we have that $\{\tau_l\le\sigma\}\subseteq \{\tau_k<\infty\}$. Therefore we get that $P(\tau_k\le \sigma)\le P(\tau_k<\infty)$. By \eqref{eq:V} and \eqref{eq:W}, we thus get
\[
\EE_P[M_{\sigma}]\le \sum_{k=1}^{\infty}w_k P(\tau_k<\infty)\le \sum_{k=1}^{\infty}w_k\alpha_k\le 1.
\]
This completes our proof.
\end{proof}

\begin{proof}[Proof of Proposition~\ref{prop:consistency-transfer}]
Denote $A:=\{\tau_k<\infty,\, \forall k\ge k_0\}$. By assumption, $Q(A)=1$ then. Now consider some $\omega\in A$. For each $n\ge k_0$ let us define $T_n(\omega):=\max_{k_0\le k\le n}\tau_k(\omega)$. Notice how only finitely many terms are involved and each $\tau_k(\omega)$ is finite on $A$. Therefore, it follows that the maximum, i.e.\ $T_n(\omega)$ is finite. Now, suppose that $t\ge T_n(\omega)$. It follows then that for every $k\in\{k_0, \dots, n\}$ we have that $\tau_k(\omega)\le t$. Therefore,
\[
M_t(\omega)=\sum_{k=1}^{\infty} w_k\ind{\tau_k(\omega)\le t}\ge \sum_{k=k_0}^n w_k.
\]
Here, the right hand side depends on $n$ rather than $t$. And by assumption as $n\to\infty$, $\sum_{k=k_0}^n w_k\to\infty$. So, for every $L>0$ there exists some $n$ such that $\sum_{k=k_0}^n w_k>L$. And then for all $t\ge T_n(\omega)$, we have $M_t(\omega)>L$. Therefore, $M_t(\omega)\to\infty$ as $t\to\infty$ since $L>0$ is arbitrary, as was $\omega\in A$. Thus this holds on all of $A$, which has $Q$ probability one, completing the proof.
\end{proof}

\begin{proof}[Proof of Lemma~\ref{lem:profile-upper-bound}]
Consider some $x\ge 0$. If $b_k\le x$, then $w_k=e^{b_k}w_k\alpha_k\le e^{x}w_k\alpha_k$. If we then sum over all indices with $b_k\le x$, we have
\[
W(x)=\sum_{b_k\le x}w_k\le e^x\sum_{b_k\le x}w_k\alpha_k\le e^{x}\sum_{k=1}^{\infty}w_k\alpha_k\le e^x.
\]
This proves our first claim. The second follows from taking logarithms and then dividing by $x$.
\end{proof}

\begin{proof}[Proof of Theorem~\ref{thm:main}]
In proving this theorem, we will first prove the following two statements.

\begin{enumerate}
    \item For every $\varepsilon\in(0,1)$ there exists an almost sure event $\Omega_\varepsilon$ such that for every $\omega\in\Omega_\varepsilon$ there is an integer $K_\varepsilon(\omega)$ for which for all $k\ge K_\varepsilon(\omega)$,
    \begin{equation}\label{eq:sandwich-tests}
    \frac{1-\varepsilon}{I}b_k\le \tau_k(\omega)\le \frac{1+\varepsilon}{I}b_k. 
    \end{equation}
    \item Take some $\varepsilon\in(0,1)$ and $\omega\in\Omega_\varepsilon$. Set $C_\varepsilon(\omega):=\sum_{k=1}^{K_\varepsilon(\omega)-1}w_k$. Then it follows that for every $t\in\mathbb N_0$,
    \begin{equation}\label{eq:M-sandwich}
    W\left(\frac{It}{1+\varepsilon}\right)-C_\varepsilon(\omega)\le M_t(\omega) \le W\left(\frac{It}{1-\varepsilon}\right)+C_\varepsilon(\omega).    
    \end{equation}
\end{enumerate}

We will first prove \textcolor{red}{statement (1)}. Consider some $\varepsilon\in(0,1)$. Thanks to \eqref{eq:AO}, the random variables $\tau_k/b_k$ converge $Q$-almost surely to $1/I$. Therefore the event
\[
\Omega_\varepsilon:=\left\{\omega : \exists K_\varepsilon(\omega)\in\NN \text{ such that } \left| \frac{\tau_k(\omega)}{b_k} - \frac{1}{I} \right| \le \frac{\varepsilon}{I} \text{ for all } k\ge K_\varepsilon(\omega) \right\}
\]
has $Q$-probability one. If we now rearrange the inequality $\left|\frac{\tau_k(\omega)}{b_k}-\frac{1}{I}\right|\le \frac{\varepsilon}{I}$, we have \eqref{eq:sandwich-tests}, proving the claim.

We will now prove \textcolor{red}{statement (2)}. Take some $\varepsilon\in(0,1)$ and $\omega\in\Omega_{\varepsilon}$. With minor abuse of notation, let us write $K:=K_{\varepsilon}(\omega)$ and $C:=C_{\varepsilon}(\omega)=\sum_{k=1}^{K-1}w_k$. We will prove the lower and upper bounds in \eqref{eq:M-sandwich} separately. We will start with the \underline{lower bound}. Consider any $k\ge K$ such that $b_k\le \frac{It}{1+\varepsilon}$. Thanks to \eqref{eq:sandwich-tests}, it follows that
\[
\tau_k(\omega)\le \frac{1+\varepsilon}{I}b_k\le \frac{1+\varepsilon}{I}\cdot\frac{It}{1+\varepsilon} = t.
\]
Thus, it follows that $\ind{\tau_k(\omega)\le t}=1$. If we then sum $w_k$ over all such $k$ we have that
\[
M_t(\omega) \ge \sum_{\substack{k\ge K\\ b_k \le It/(1+\varepsilon)}} w_k.
\]
We will now split the entire profile $W(It/(1+\varepsilon))$ into indices below $K$ and indices at least $K$. That is,
\[
W\left(\frac{It}{1+\varepsilon}\right) = \sum_{\substack{k< K\\ b_k \le It/(1+\varepsilon)}} w_k + \sum_{\substack{k\ge K\\ b_k \le It/(1+\varepsilon)}} w_k.
\]
Therefore,
\[
\sum_{\substack{k\ge K\\ b_k \le It/(1+\varepsilon)}} w_k = W\left(\frac{It}{1+\varepsilon}\right)- \sum_{\substack{k< K\\ b_k \le It/(1+\varepsilon)}} w_k \ge W\left(\frac{It}{1+\varepsilon}\right) - C.
\]
This holds because the subtracted sum will be at most $\sum_{k<K}w_k=C$. Therefore, $M_t(\omega)\ge W\left(\frac{It}{1+\varepsilon}\right)-C$. We will now prove the \underline{upper bound}. To this end, suppose that $k\ge K$ and $\tau_k(\omega)\le t$. By the lower half of \eqref{eq:sandwich-tests}, we have that
\[
\frac{1-\varepsilon}{I}b_k\le \tau_k(\omega)\le t.
\]
If we multiply by $I/(1-\varepsilon)$, we have that $b_k\le \frac{It}{1-\varepsilon}$. Therefore, each index $k\ge K$ that contributes to $M_t(\omega)$ must satisfy $b_k\le It/(1-\varepsilon)$. Therefore, we have
\[
M_t(\omega) \le \sum_{k< K} w_k + \sum_{\substack{k\ge K\\ b_k \le It/(1-\varepsilon)}} w_k \le C + W\left(\frac{It}{1-\varepsilon}\right).
\]
This proves \eqref{eq:M-sandwich}. We will now prove the \textcolor{red}{main statement} of the theorem. Assume to this end that \eqref{eq:R} holds and that $W(x)\to\infty$. Consider some $\varepsilon\in(0,1)$ and let us work on the same full probability event $\Omega_{\varepsilon}$ from the proof of \textcolor{red}{(1)}. Once again, we will write $C=C_{\varepsilon}(\omega)$. Since $W(x)\to\infty$, there must exist some $t_0(\omega)$ such that for every $t\ge t_0(\omega)$, both $W\left(\frac{It}{1+\varepsilon}\right)>2C$ and $W\left(\frac{It}{1-\varepsilon}\right)>C$. Now, for such a $t$, thanks to the lower bound in \eqref{eq:M-sandwich}, we have
\[
M_t(\omega)\ge W\left(\frac{It}{1+\varepsilon}\right)-C\ge \frac{1}{2}W\left(\frac{It}{1+\varepsilon}\right).
\]
On a similar vein, the upper bound implies that
\[
M_t(\omega)\le W\left(\frac{It}{1-\varepsilon}\right)+C\le 2W\left(\frac{It}{1-\varepsilon}\right).
\]
Let us now take the logarithms and divide by $t$. We then have for all $t$ sufficiently large enough that
\[
\frac{1}{t}\log W\left(\frac{It}{1+\varepsilon}\right) - \frac{\log 2}{t} \le \frac{1}{t}\log M_t(\omega) \le \frac{1}{t}\log W\left(\frac{It}{1-\varepsilon}\right) + \frac{\log 2}{t}.
\]
Let us use \eqref{eq:R}. For the lower bound, we have
\[
\frac{1}{t}\log W\left(\frac{It}{1+\varepsilon}\right) = \frac{I}{1+\varepsilon} \cdot \frac{1}{\,It/(1+\varepsilon)\,} \log W\left(\frac{It}{1+\varepsilon}\right) \longrightarrow \frac{\rho I}{1+\varepsilon}.
\]
Similarly, 
\[
\frac{1}{t}\log W\left(\frac{It}{1-\varepsilon}\right)\longrightarrow \frac{\rho I}{1-\varepsilon}.
\]
Hence it follows that $\liminf_{t\to\infty} \frac{1}{t}\log M_t(\omega) \ge \frac{\rho I}{1+\varepsilon}$ and $\limsup_{t\to\infty} \frac{1}{t}\log M_t(\omega) \le \frac{\rho I}{1-\varepsilon}$. This will hold for all $\varepsilon\in(0,1)$ on the corresponding almost sure event $\Omega_{\varepsilon}$. We can now intersect over the countable set $\varepsilon\in \mathbb  Q\cap(0,1)$ to get a new event of $Q$-probability one on which both inequalities hold simultaneously for each rational $\varepsilon\in(0,1)$. We can now let $\varepsilon\downarrow 0$ through the rationals to give us
\[
\liminf_{t\to\infty}\frac{1}{t}\log M_t(\omega)\ge \rho I,\quad \limsup_{t\to\infty}\frac{1}{t}\log M_t(\omega)\le \rho I.
\]
The limit therefore exists and equals $\rho I$. This proves \eqref{eq:G}, completing the proof.
\end{proof}

\begin{proof}[Proof of Theorem~\ref{thm:main-prob}]
Take some $\eta>0$ and $\delta\in(0,1)$. We will choose $\varepsilon\in(0,1)$ small enough so that $\left|\frac{\rho I}{1+\varepsilon}-\rho I\right|<\frac{\eta}{3}$ and $\left|\frac{\rho I}{1-\varepsilon}-\rho I\right|<\frac{\eta}{3}$. Thanks to our tail uniform optimality assumption, it follows that there exist $K=K(\varepsilon,\delta)$ and an event $A_{\varepsilon,\delta}$ where $Q(A_{\varepsilon,\delta})\ge 1-\delta$ such that on $A_{\varepsilon, \delta}$ for all $k\ge K$,
\[
\frac{1-\varepsilon}{I}b_k\le \tau_k \le \frac{1+\varepsilon}{I}b_k.
\]
All this is just by definition of tail-uniform optimality. Let us now set $C:=\sum_{k=1}^{K-1} w_k<\infty$. As in our proof of the second part of Theorem~\ref{thm:main}, on the event $A_{\varepsilon,\delta}$ we have that for every $t\in\mathbb N_0$,
\[
W\left(\frac{It}{1+\varepsilon}\right)-C\le M_t\le W\left(\frac{It}{1-\varepsilon}\right)+C.
\]
Now, we know that $W(x)\to\infty$. It follows that there must exist some $t_1$ such that for every $t\ge t_1$, $W\left(\frac{It}{1+\varepsilon}\right)>2C$ and $W\left(\frac{It}{1-\varepsilon}\right)>C$. Therefore, it follows that on $A_{\varepsilon,\delta}$ for all $t\ge t_1$, $\frac12W\left(\frac{It}{1+\varepsilon}\right)\le M_t\le2W\left(\frac{It}{1-\varepsilon}\right)$. Now, using \eqref{eq:R}, let us choose $t_2\ge t_1$ large enough so that for all $t\ge t_2$, both $\left| \frac{1}{t}\log W\left(\frac{It}{1+\varepsilon}\right) - \frac{\rho I}{1+\varepsilon} \right| < \frac{\eta}{3}$ and $\left| \frac{1}{t}\log W\left(\frac{It}{1-\varepsilon}\right) - \frac{\rho I}{1-\varepsilon} \right| < \frac{\eta}{3}$, and $\frac{\log 2}{t}<\frac{\eta}{3}$. Thus, on $A_{\varepsilon,\delta}$, we have for all $t\ge t_2$ that
\[
\frac{1}{t}\log M_t \ge \frac{1}{t}\log W\left(\frac{It}{1+\varepsilon}\right)-\frac{\log 2}{t} > \frac{\rho I}{1+\varepsilon}-\frac{2\eta}{3} > \rho I-\eta.
\]
Similarly,
\[
\frac{1}{t}\log M_t \le \frac{1}{t}\log W\left(\frac{It}{1-\varepsilon}\right)+\frac{\log 2}{t} < \frac{\rho I}{1-\varepsilon}+\frac{2\eta}{3} < \rho I+\eta.
\]
Therefore, it must hold that for every $t\ge t_2$,
\[
Q\left( \left| \frac{1}{t}\log M_t-\rho I \right|>\eta \right) \le Q(A_{\varepsilon,\delta}^c) \le \delta.
\]
Recall that $\delta\in(0,1)$ was chosen arbitrarily. This proves convergence in $Q$-probability, completing the proof.
\end{proof}

\section{Omitted Proofs for the Profile Asymptotics}

\begin{proof}[Proof of Corollary~\ref{cor:unweighted}]
This is the same as Theorem~\ref{thm:main} applied with $w_k\equiv 1$. Because then $W(x)=N(x)$.    
\end{proof}

\begin{proof}[Proof of Lemma~\ref{lem:inversion}]
We will first prove \textcolor{red}{(1)}. Consider some $\delta\in(0,c)$. Because $b_k/\log k\to c$, it follows that there exists some $k_{\delta}$ such that for all $k\ge k_{\delta}$,
\begin{equation}\label{eq:I1}
(c-\delta)\log k\le b_k\le (c+\delta)\log(k).  
\end{equation}
We will first derive a lower bound on $N(x)$. Clearly, if $k\ge k_{\delta}$ and $(c+\delta)\log k\le x$, then thanks to \eqref{eq:I1}, we have that $b_k\le x$, therefore meaning $k\le N(x)$. Equivalently, we can recognize that every integer $k\le \exp\left(\frac{x}{c+\delta}\right)$ with $k\ge k_{\delta}$ will contribute to $N(x)$. Therefore, for all $x$ sufficiently large, we have that 
\[
N(x)\ge \exp\left(\frac{x}{c+\delta}\right)-k_{\delta}.
\]
If we now take logarithms and divide by $x$, we have
\[
\liminf_{x\to\infty}\frac{\log N(x)}{x}\ge\frac{1}{c+\delta}.
\]
We will now derive an upper bound. If $b_k\le x$ and $k\ge k_{\delta}$, then thanks to \eqref{eq:I1}, we have that $(c-\delta)\log k\le b_k\le x$. Therefore, $k\le \exp\left(\frac{x}{c-\delta}\right)$. Therefore, all the contributing indices will satisfy this bound save for possibly finitely many indices $k<k_{\delta}$. Thus for all $x$ large enough, we have
\[
N(x)\le k_{\delta} + \exp\left(\frac{x}{c-\delta}\right).
\]
Taking logarithms and dividing by $x$ gives us
\[
\limsup_{x\to\infty}\frac{\log N(x)}{x}\le \frac{1}{c-\delta}.
\]
Now, $\delta\in(0,c)$ was arbitrary, hence letting $\delta\downarrow 0$, we may conclude that $\log N(x)/x \to 1/c$.

We will now prove \textcolor{red}{(2)}. Consider any $A>0$. Because $b_k/\log k\to \infty$, it follows that there exists some $k_A$ such that $b_k\ge A\log k$ for all $k\ge k_A$. Now, if $b_k\le x$ and $k\ge K_A$, then it must hold that $A\log k\le x$. In other words, $k\le e^{x/A}$. Therefore, for all $x$ large enough, we have that $N(x)\le k_A + e^{x/A}$. Therefore,
\[
\limsup_{x\to\infty}\frac{\log N(x)}{x}\le \frac{1}{A}.
\]
Now, recognizing that $A>0$ was arbitrary and letting $A\to\infty$, we have
\[
\limsup_{x\to\infty}\frac{\log N(x)}{x}\le 0.
\]
Since $N(x)\ge 1$ for all $x$ large enough, the left hand side above is nonnegative. Thus, the limit exists and equals $0$. This completes the proof.
\end{proof}

\begin{proof}[Proof of Corollary~\ref{cor:classification}]
The first statement is immediate from the first statement of Lemma~\ref{lem:inversion} and Corollary~\ref{cor:unweighted}. The second statement here follows from the second statement of Lemma~\ref{lem:inversion} and Corollary~\ref{cor:unweighted}. Finally, the third statement is simply a special case (the case where $c=1$) of the first statement.   
\end{proof}

\begin{proof}[Proof of Proposition~\ref{prop:general-h}]
We must first show that the level schedule is decreasing. By assumption, $h$ is non-decreasing. Thus for $k\ge k_0$ we have that
\[
\frac{\alpha_{k+1}}{\alpha_k}=\frac{k}{k+1}\exp\bigl(-(h(k+1)-h(k))\bigr)\le \frac{k}{k+1}<1.
\]
It follows therefore that $(\alpha_k)_{k\ge k_0}$ is strictly decreasing. Thus, $b_k=\log(1/\alpha_k)$ is increasing. We now show $e$-process validity. By definition of $c_h$, we have that
\[
\sum_{k=k_0}^\infty \alpha_k = c_h \sum_{k=k_0}^\infty \frac{e^{-h(k)}}{k} = 1.
\]
Now let $\sigma$ be any finite stopping rule and consider some $P\in\PP$. Thanks to Tonelli's theorem and \eqref{eq:V}, 
\[
\EE_P[B_\sigma^{(h)}] = \sum_{k=k_0}^\infty P(\tau_{\alpha_k}\le \sigma)\le\sum_{k=k_0}^\infty P(\tau_{\alpha_k}<\infty)\le\sum_{k=k_0}^\infty \alpha_k = 1.
\]
It follows that $B^{(h)}$ is a $\PP$-$e$-process. Now, $b_k=\log(1/\alpha_k)=\log k + h(k)-\log c_h$. Dividing by $\log k$ and sending $k\to\infty$ thus gives us
\[
\frac{b_k}{\log k}=1+\frac{h(k)}{\log k}-\frac{\log c_h}{\log k} \longrightarrow 1.
\]
Because, $h(k)=o(\log k)$ and $(\log c_h)/\log k\to 0$. Now, recall that we assumed asymptotic optimality along all these levels. We may also freely reindex $\{k_0, k_0+1,\dots\}$ to $\mathbb N$. All these allow us to conclude thanks to Corollary~\ref{cor:classification}'s first statement with $c=1$, that $\log (B_t^{(h)})/t\to I$ almost surely under $Q$. This completes the proof.
\end{proof}

\begin{proposition}\label{prop:dyadic-count}
Suppose that $\alpha_k=2^{-k}$ and that $E_t:=\sum_{k=1}^\infty \ind{\tau_{2^{-k}}\le t}$. Assume that wp 1 under $Q$ that,
\[
\frac{\tau_{2^{-k}}}{k\log 2}\longrightarrow \frac{1}{I}.
\]
Then it follows that $E$ is a $\mathcal P$-$e$-process. And wp 1 under $Q$,
\[
\frac{1}{t}\log E_t\longrightarrow 0.
\]
\end{proposition}

\begin{proof}
Thanks to Proposition~\ref{prop:eprocess-validity}, we have validity since with $w_k\equiv 1$,
\[
\sum_{k=1}^{\infty}\alpha_k=\sum_{k=1}^{\infty}2^{-k}=1.
\]
Now, for the growth rate, note that $b_k=\log(1/\alpha_k)=k\log 2$. Thus,
\[
\frac{b_k}{\log k}=\frac{k\log 2}{\log k}\to\infty.
\]
Thus, thanks to the second statement of Corollary~\ref{cor:classification} we can conclude that $\frac{\log E_t}{t}\to 0$ with probability one under $Q$. This completes the proof.
\end{proof}

Notice how in Proposition~\ref{prop:dyadic-count} we have the issue where at the time scale $t\asymp k/I$ on which this level $2^{-k}$ rejects, the counting process will only have size about $k$, and thus $\log k = o(k)$. Now, the next idea here is that power law schedules give rate $I/(1+\varepsilon)$.

\begin{proposition}\label{prop:power-law}
Consider some $\varepsilon>0$. And let
\[
c_\varepsilon:=\frac{1}{\sum_{k=1}^\infty k^{-(1+\varepsilon)}}=\frac{1}{\zeta(1+\varepsilon)}.
\]
Now let us define $\alpha_k:=c_\varepsilon k^{-(1+\varepsilon)}$. And, $B_t^{(\varepsilon)}:=\sum_{k=1}^\infty\ind{\tau_{\alpha_k}\le t}$. Assume that wp 1 under $Q$,
\[
\frac{\tau_{\alpha_k}}{\log(1/\alpha_k)} \longrightarrow \frac{1}{I}.
\]
Then it follows that $B^{(\varepsilon)}$ is a $\mathcal P$-$e$-process. And wp 1 under $Q$,
\[
\frac{1}{t}\log B_t^{(\varepsilon)}\longrightarrow \frac{I}{1+\varepsilon}.
\]
\end{proposition}

\begin{proof}
By construction, we have
\[
\sum_{k=1}^{\infty}\alpha_k = c_{\varepsilon}\sum_{k=1}^{\infty}k^{-(1+\varepsilon)}=1,
\]
proving that $B^{(\varepsilon)}$ is a $\PP$-$e$-process thanks to Proposition~\ref{prop:eprocess-validity}. Also, $b_k=\log(1/\alpha_k)=(1+\varepsilon)\log k-\log c_{\varepsilon}$. Therefore,
\[
\frac{b_k}{\log k}=1+\varepsilon-\frac{\log c_{\varepsilon}}{\log k}\longrightarrow 1+\varepsilon.
\]
As such, thanks to the first statement of Corollary~\ref{cor:classification} with $c=1+\varepsilon$, we have almost surely under $Q$ that
\[
\frac{1}{t}\log B_t^{\varepsilon}\longrightarrow\frac{I}{1+\varepsilon}.
\]
\end{proof}

Another beautiful idea is if we consider a full-rate unweighted schedule. Meaning, schedules with $1/[k(\log k)^p]$. We will formalize these $\log$-$p$ schedules in the following corollary.

\begin{corollary}\label{cor:logp}
Take some $p>1$ and some particular index $k_0\ge 3$. Where $k\ge k_0$ let us define
\[
S_p:=\sum_{k=k_0}^\infty \frac{1}{k(\log k)^p},
\]
And, in addition, let us define $c_p:=1/S_p$, and $\alpha_k:=\frac{c_p}{k(\log k)^p}$. Now, let $B_t^{(p)}:=\sum_{k=k_0}^\infty \ind{\tau_{\alpha_k}\le t}$. Assume that wp 1 under $Q$,
\[
\frac{\tau_{\alpha_k}}{\log(1/\alpha_k)} \longrightarrow \frac{1}{I}.
\]
Then $S_p<\infty$ and $B^{(p)}$ is a $\mathcal P$-$e$-process. And wp 1 under $Q$,
\[
\frac{1}{t}\log B_t^{(p)}\longrightarrow I.
\]
\end{corollary}

\begin{proof}
The first thing we need to show is that $S_p<\infty$. Observe that the function $f_p(x):=\frac{1}{x(\log x)^p}$ for $x\ge 3$ is positive and decreasing for large $x$, meaning we may apply the integral test. If we substitute $u=\log x$, it is clear that $\int \frac{dx}{x(\log x)^p}=\int u^{-p}du$. Therefore,
\[
\int_{k_0}^{\infty} \frac{dx}{x(\log x)^p} = \int_{\log k_0}^{\infty} u^{-p}du = \frac{(\log k_0)^{1-p}}{p-1}<\infty.
\]
It follows that $S_p<\infty$. Now, observe that $\sum_{k=k_0}^{\infty}\alpha_k=c_p S_p=1$. Thus, $B^{(p)}$ must be a $\PP$-$e$-process thanks to Proposition~\ref{prop:eprocess-validity}. Now, let us compute $b_k$. That is, $b_k=\log(1/\alpha_k)=\log k + p\log\log k - \log c_p$. Dividing by $\log k$, we have
\[
\frac{b_k}{\log k} = 1 + p\frac{\log \log k}{\log k}-\frac{\log c_p}{\log k}.
\]
Both $\log\log k/\log k\to 0$ and $(\log c_p)/\log k\to 0$. So, $b_k/\log k\to 1$. Through the first statement of Corollary~\ref{cor:classification} with $c=1$ we thus get with probability 1 under $Q$,
\[
\frac{1}{t}\log B_t^{(p)}\longrightarrow I.
\]
This completes the proof.
\end{proof}

Note that really we can take $p=2$, i.e.\ it is the simplest concrete choice. That is, take $\alpha_k\propto \frac{1}{k(\log k)^2}$. Now, the last schedule we will present is one where we have $1/[k\log k(\log\log k)^2]$.

\begin{corollary}\label{cor:gamma}
Take some index $k_0\ge 16$. For $k\ge k_0$ let us define $S_\star:=\sum_{k=k_0}^\infty\frac{1}{k\log k(\log\log k)^2}$, $c^\star:=1/S_\star$. And,
\[
\gamma_k:=\frac{c_\star}{k\log k(\log\log k)^2}.
\]
Now let $G_t:=\sum_{k=k_0}^\infty \ind{\tau_{\gamma_k}\le t}$. And assume that wp 1 under $Q$,
\[
\frac{\tau_{\gamma_k}}{\log(1/\gamma_k)}\longrightarrow \frac{1}{I}.
\]
Then it follows that $S_\star<\infty$ and $G$ is a $\mathcal P$-$e$-process. And wp 1 under $Q$,
\[
\frac{1}{t}\log G_t\longrightarrow I.
\]
\end{corollary}

\begin{proof}
The first thing we will show is that $S_{\star}<\infty$. For $x\ge 16$, the function $f_{\star}(x):=\frac{1}{x\log x(\log\log x)^2}$ is positive and for large $x$ it is decreasing. Let us now make the substitutions of $u=\log x$ and $v=\log u=\log\log x$. Therefore,
\[
\int \frac{dx}{x\log x(\log\log x)^2}= \int \frac{du}{u(\log u)^2} = \int v^{-2}dv =-\frac{1}{v}=-\frac{1}{\log\log x}.
\]
And thus,
\[
\int_{k_0}^{\infty} \frac{dx}{x\log x(\log \log x)^2}=\frac{1}{\log \log k_0}<\infty.
\]
By the integral test, $S_{\star}<\infty$. And so $\sum_{k=k_0}^{\infty} \gamma_k=c_{\star}S_{\star}=1$. It follows that $G$ is a $\PP$-$e$-process thanks to Proposition~\ref{prop:eprocess-validity}. Let us now compute the $\log$-level. In particular, $\log(1/\gamma_k)=\log k + \log\log k + 2\log\log\log k - \log c_{\star}$. Dividing by $\log k$, we have
\[
\frac{\log(1/\gamma_k)}{\log k}=1+\frac{\log\log k}{\log k} + 2\frac{\log\log\log k}{\log k}-\frac{\log c_{\star}}{\log k}.
\]
As $k\to\infty$, each of these latter three terms goes to $0$. So, $\log(1/\gamma_k)/\log k \to 1$. Once again, thanks to the first statement of Corollary~\ref{cor:classification} with $c=1$, we have that $\log(G_t)/ t \to I$ almost surely under $Q$. This completes the proof.
\end{proof}

\begin{proof}[Proof of Proposition~\ref{prop:weighted-dyadic}]
Let us first show validity. To this end, set $w_k:=\frac{6}{\pi^2}\frac{2^k}{k^2}$ and $\alpha_k:=2^{-k}$. Then it follows that
\[
\sum_{k=1}^{\infty} w_k\alpha_k =\frac{6}{\pi^2}\sum_{k=1}^{\infty}\frac{1}{k^2} = \frac{6}{\pi^2}\cdot\frac{\pi^2}{6}=1.
\]
Thanks to Proposition~\ref{prop:eprocess-validity}, it follows that $M^{\mathrm{dy}}$ is a $\PP$-$e$-process. Now, here we have that $b_k=\log(1/\alpha_k)=k\log 2$. Therefore, for $x\ge0$ we have
\[
W_{\mathrm{dy}}(x) = \frac{6}{\pi^2} \sum_{k:\, k\log 2 \le x} \frac{2^k}{k^2} = \frac{6}{\pi^2} \sum_{k=1}^{\lfloor x/\log 2\rfloor} \frac{2^k}{k^2}.
\]
Therefore, everything here will reduce to determining the asymptotics of the partial sums $S_n:=\sum_{k=1}^n \frac{2^k}{k^2}$. We will show that this $S_n$ is comparable to its last term. To this end, let $a_k:=\frac{2^k}{k^2}$. For $k\ge 5$, we have
\[
\frac{a_{k-1}}{a_k} = \frac{2^{k-1}/(k-1)^2}{2^k/k^2} = \frac12 \left(\frac{k}{k-1}\right)^2 \le \frac12 \left(\frac54\right)^2 = \frac{25}{32} < \frac45.
\]
Therefore, it follows that for every $n\ge 5$ and every $j\in\{0,1,\dots,n-5\}$, we have that $a_{n-j}\le \left(\frac45\right)^{j}a_n$. If we sum over all $j$, we have
\[
\sum_{k=5}^n a_k=\sum_{j=0}^{n-5} a_{n-j}\le a_n \sum_{j=0}^{\infty}\left(\frac45\right)^j = 5a_n.
\]
Therefore, 
\[
S_n = \sum_{k=1}^4 a_k + \sum_{k=5}^n a_k \le \left(\sum_{k=1}^4 a_k\right)+5a_n.
\]
Now, $a_n=2^{n}/n^{2}\to\infty$. Therefore, there exists some $n_0$ such that for all $n\ge n_0$, $\sum_{k=1}^4 a_k\le a_n$. Thus, for all $n\ge n_0$, we have $a_n\le S_n\le 6a_n$. Equivalently we can state that for all $n\ge n_0$,
\begin{equation}\label{eq:Dy}
\frac{2^n}{n^2}\le S_n\le6\frac{2^n}{n^2}.
\end{equation}
We will now extract the exponential rate of $W_{\mathrm{dy}}$. Let $n(x):=\lfloor x/\log 2\rfloor$. Then, $n(x)\to\infty$ and $W_{\mathrm{dy}}(x)=\frac{6}{\pi^2}S_{n(x)}$. Thanks to \eqref{eq:Dy}, we have
\[
\log W_{\mathrm{dy}}(x) = \log\left(\frac{6}{\pi^2}\right) + \log S_{n(x)} = n(x)\log 2 - 2\log n(x) + O(1).
\]
Now, $n(x)\log 2=x+O(1)$. And, $\log n(x)=o(x)$. We may therefore conclude that $\log W_{\mathrm{dy}}(x)/x \to 1$. Now, recall we assumed that $\tau_{2^{-k}}/b_k\to 1/I$ almost surely under $Q$. Further, we proved that the profile satisfies $\log W_{\mathrm{dy}}(x)/x \to 1$. Therefore, thanks to Theorem~\ref{thm:main}, we may conclude that almost surely under $Q$,
\[
\frac1t\log M_t^{\mathrm{dy}}\longrightarrow I.
\]
This completes the proof.
\end{proof}

\section{Omitted Proofs for Optimality Implications}

\begin{proof}[Proof of Proposition~\ref{prop:log-growth-to-threshold}]
Let us write $b=\log(1/\alpha)$. For this, we will show the almost sure case. Note that the in probability case is identical upon replacing the eventual inequalities with probability bounds (and using monotonicity). Consider some $\eta\in(0,1)$. Consider any path where $\log(E_t)/t\to L$. On such a path, for all sufficiently large $t$, we have $(1-\eta)Lt\le \log E_t\le (1+\eta)Lt$. Now, for $t_{+}(b)=\lceil \frac{b}{(1-\eta)L}\rceil$, our lower bound tells us that $\log E_{t_{+}(b)}\ge b$ for all large enough $b$. It follows that $T_{\alpha}(E)\le t_{+}(b)$. Now, for $t_{-}(b):=\lfloor\frac{b}{(1+\eta)L}\rfloor-1$, thanks to both monotonicity and the upper bound we have for all large enough $b$ that
\[
\sup_{t\le t_{-}(b)}\log E_t=\log E_{t_{-}(b)}\le (1+\eta)Lt_{-}(b)<b.
\]
Therefore, $T_{\alpha}(E)>t_{-}(b)$ holds. As a result,
\[
\frac{1}{(1+\eta)L}\le\liminf_{\alpha\downarrow0}\frac{T_\alpha(E)}{\log(1/\alpha)}\le\limsup_{\alpha\downarrow0}\frac{T_\alpha(E)}{\log(1/\alpha)}\le\frac{1}{(1-\eta)L}.
\]
Letting $\eta\downarrow 0$ completes the proof.
\end{proof}

\begin{proof}[Proof of Theorem~\ref{thm:nested-inprob}]
Consider some $\varepsilon\in(0,1)$. For $x\ge b_1$, let us define $k_{-}(x):=\max\{k\in\mathbb N: b_k\le x\}$ and $k_+(x):=\min\{k\in\NN : b_k>x\}=k_-(x)+1$. Since $b_k\uparrow\infty$, it follows that both indices will be well defined for all sufficiently large $x$, and $k_{\pm}(x)\to\infty$ as $x\to\infty$. Now let us set $x_t^{-}:=\frac{It}{1+\varepsilon}$ and $x_t^{+}:=\frac{It}{1-\varepsilon}$. We will first derive a lower bound that holds with high $Q$-probability. If $\tau_{k_{-}(x_t^{-})}\le t$, then thanks to nestedness we have that for every $j\le k_{-}(x_t^{-})$, $\tau_j\le t$. Therefore, it follows that
\[
M_t\ge \sum_{j=1}^{k_-(x_t^-)} w_j=W(x_t^-).
\]
Therefore, it follows that we must have $Q(M_t\ge W(x_t^{-}))\ge Q(\tau_{k_{-}(x_t^{-})}\le t)$. Also, $b_{k_{-}(x_t^{-})}\le x_t^{-}=It/(1+\varepsilon)$. So, $t\ge \frac{1+\varepsilon}{I}b_{k_{-}(x_t^{-})}$. Therefore it must hold that
\[
Q\bigl(\tau_{k_-(x_t^-)} \le t\bigr) \ge Q\left( \frac{\tau_{k_-(x_t^-)}}{b_{k_-(x_t^-)}} \le \frac{1+\varepsilon}{I} \right) \longrightarrow 1.
\]
This follows from the fact that $k_{-}(x_t^{-})\to\infty$ and that coordinate-wise convergence in $Q$-probability will be preserved along these deterministic subsequences. 

We now move on to the upper bound. If $\tau_{k_{+}(x_t^{+})}>t$, then thanks to nestedness, for every $j\ge k_{+}(x_t^{+})$, $\tau_j>t$. This means that only the indices $j<k_{+}(x_t^{+})$ can contribute to $M_t$. Therefore,
\[
M_t\le \sum_{j<k_{+}(x_t^{+})} w_j = W(x_t^{+}).
\]
It follows that $Q(M_t\le W(x_t^{+}))\ge Q(\tau_{k_{+}(x_t^{+})}>t)$. Now, observe that $b_{k_{+}(x_t^{+})}>x_t^{+}=It/(1-\varepsilon)$. Thus, we have $t<\frac{1-\varepsilon}{I}b_{k_{+}(x_t^{+})}$. Hence,
\[
Q\bigl(\tau_{k_+(x_t^+)} > t\bigr) \ge Q\left( \frac{\tau_{k_+(x_t^+)}}{b_{k_+(x_t^+)}} > \frac{1-\varepsilon}{I} \right) \longrightarrow 1.
\]
Together, both estimates give us
\[
Q\left( W\left(\frac{It}{1+\varepsilon}\right) \le M_t \le W\left(\frac{It}{1-\varepsilon}\right) \right) \longrightarrow 1.
\]
We may now conclude that $\log(M_t)/t \to \rho I$ in $Q$-probability by our squeezing argument from Theorem~\ref{thm:main-prob}. This completes the proof.
\end{proof}

\begin{proof}[Proof of Corollary~\ref{cor:expectation-plus-concentration}]
We will first prove \textcolor{red}{(1)}. Thanks to Markov's inequality for every $\varepsilon>0$ we have
\[
Q\left(\left|\frac{\tau_k}{b_k}-\frac{1}{I}\right|>\varepsilon\right)\le \frac{1}{\varepsilon b_k}\EE_Q\left[\left|\tau_k - \frac{b_k}{I}\right|\right]\longrightarrow 0.
\]
Therefore, $\tau_k/b_k\to 1/I$ in $Q$-probability, so we may apply Theorem~\ref{thm:nested-inprob} to complete the proof for this case. Now, we prove case \textcolor{red}{(2)}. Because $\EE_Q[\tau_k]/b_k\to 1/I$, there exists $K_{\varepsilon}$ such that for every $k\ge K_{\varepsilon}$,
\[
\left|\frac{\EE_Q[\tau_k]}{b_k}-\frac{1}{I}\right|\le \frac{\varepsilon}{2}.
\]
For such $k$, thanks to Chebyshev's inequality we have
\[
Q\left(\left|\frac{\tau_k}{b_k}-\frac{1}{I}\right|>\varepsilon\right)\le Q\left(|\tau_k-\EE_Q[\tau_k]|>\frac{\varepsilon b_k}{2}\right)\le \frac{4\operatorname{Var}_Q(\tau_k)}{\varepsilon^2 b_k^2}.
\]
The right hand side here tends to $0$ by assumption. Therefore, $\tau_k/b_k\to 1/I$ in $Q$-probability, meaning Theorem~\ref{thm:nested-inprob} once again applies. This completes the proof.
\end{proof}

\begin{proof}[Proof of Corollary~\ref{cor:summable-deviations}]
Consider some rational $\varepsilon\in(0,1)$. Thanks to the Borel-Cantelli Lemma, it holds that
\[
Q\left(\left|\frac{\tau_k}{b_k}-\frac{1}{I}\right|>\frac{\varepsilon}{I}\text{ i.o.}\right)=0.
\]
Let us now intersect over the countable set of rational $\varepsilon\in(0,1)$. Then, for $Q$-almost every $\omega$, for every rational $\varepsilon$, there exists some $K_{\varepsilon}(\omega)$ such that for all $k\ge K_{\varepsilon}(\omega)$,
\[
\left|\frac{\tau_k(\omega)}{b_k}-\frac{1}{I}\right|\le\frac{\varepsilon}{I}.
\]
This is \eqref{eq:AO}, completing the proof.
\end{proof}

\section{How to Build First-Passage Tests from Score Processes}\label{sec:extremelycute}
We will show how, concretely, one can construct such almost sure optimal sequential tests. One example of this is as follows. Suppose we have a score process that grows linearly. Then first-passage tests from this will be pathwise optimal.

\begin{proposition}\label{prop:score-to-as}
Suppose that $S=(S_t)_{t\in\NN_0}$ is some adapted real-valued process. Assume that for every $x\ge 0$,
\[
\sup_{P\in\mathcal P}P\left(\sup_{t\in\NN_0}S_t\ge x\right)\le e^{-x}.
\]
Now, for $\alpha\in(0,1)$, let us define the first passage time $\tau_\alpha:=\inf\{t\in\NN_0: S_t\ge \log(1/\alpha)\}$. Then it follows that for every $\alpha\in(0,1)$, $\sup_{P\in\mathcal P}P(\tau_\alpha<\infty)\le \alpha$. Now suppose in addition we have that under $Q$ almost surely as $t\to\infty$ for some $I\in(0,\infty)$,
\[
\frac{S_t}{t}\longrightarrow I.
\]
Then it follows that for every sequence $\alpha_k\downarrow 0$ with $b_k=\log(1/\alpha_k)$ we have that almost surely under $Q$,
\[
\frac{\tau_{\alpha_k}}{b_k}\longrightarrow\frac{1}{I}.
\]
\end{proposition}

\begin{proof}
Notice how we have that $\{\tau_{\alpha}<\infty\}\subseteq\{\sup_{t\in\mathbb N_0}S_t\ge \log(1/\alpha)\}$. Therefore,
\[
\sup_{P\in\PP}P(\tau_\alpha<\infty)\le \exp(-\log(1/\alpha))=\alpha.
\]
Consider an $\omega$ such that $S_t(\omega)/t \to I$. Now let $\varepsilon\in(0,I)$. Then there must exist some $T_{\varepsilon}(\omega)$ such that for all $t\ge T_{\varepsilon}(\omega)$, $(I-\varepsilon)t\le S_t(\omega)\le(I+\varepsilon)t$. Since $b_k\to\infty$, there must exist some $K_{\varepsilon}(\omega)$ such that $b_k> \max_{0\le t < T_{\varepsilon}(\omega)} S_t(\omega)$ for all $k\ge K_{\varepsilon}(\omega)$. Now, for such a $k$, the crossing level $b_k$ cannot be reached before time $T_{\varepsilon}(\omega)$. Therefore, $\tau_{\alpha_k}(\omega)\ge T_{\varepsilon}(\omega)$. 

Now, for the upper bound, let us define $u_k:=\lceil \frac{b_k}{I-\varepsilon}\rceil$. We then have that $u_k\ge T_{\varepsilon}(\omega)$ for all large $k$. Therefore, $S_{u_k}(\omega)\ge (I-\varepsilon)u_k\ge b_k$. Thus, it follows that $\tau_{\alpha_k}(\omega)\le u_k$. Thus,
\[
\frac{\tau_{\alpha_k}(\omega)}{b_k}\le \frac{1}{I-\varepsilon}+\frac{1}{b_k}.
\]
Now, for the lower bound, we have already shown for all large $k$, $\tau_{\alpha_k}(\omega)\ge T_{\varepsilon}(\omega)$. Therefore, $b_k\le S_{\tau_{\alpha_k}(\omega)}(\omega)\le(I+\varepsilon)\tau_{\alpha_k}(\omega)$. And from this, we can see that
\[
\frac{\tau_{\alpha_k}(\omega)}{b_k}\ge \frac{1}{I+\varepsilon}.
\]
Let us now take both estimates together. For all large $k$, we therefore have
\[
\frac{1}{I+\varepsilon}\le \frac{\tau_{\alpha_k}(\omega)}{b_k}\le \frac{1}{I-\varepsilon}+\frac{1}{b_k}.
\]
If we send $k\to\infty$ and then $\varepsilon\downarrow 0$ through the rationals we get that $\tau_{\alpha_k}(\omega)/b_k \to 1/I$. This holds on the almost sure event $\{S_t/t\to I\}$, proving the claim.
\end{proof}

Proposition~\ref{prop:score-to-as} applies to the following constructions. First, suppose that $E=(E_t)$ is a strictly positive $\mathcal P$-$e$-process. And suppose that one sets $S_t:=\log E_t$. Then by composite Ville's inequality for e-processes it follows that
\[
\sup_{P\in\PP}P\left(\sup_{t\in\Nzero}S_t\ge x\right)=\sup_{P\in\PP} P\left(\sup_{t\in\Nzero}E_t\ge e^x\right)\le e^{-x}.
\]
Then, it follows that the threshold test $\tau_\alpha:=\inf\{t\in\NN_0: E_t\ge 1/\alpha\}$ is valid. And in addition suppose almost surely under $Q$ we have that
\[
\frac{1}{t}\log E_t\longrightarrow I.
\]
Then it follows that almost surely under $Q$,
\[
\frac{\tau_{\alpha_k}}{\log(1/\alpha_k)}\longrightarrow \frac{1}{I}.
\]

In the setting where we have a simple null versus simple alternative, we can take $S_t:=\sum_{i=1}^t \log\frac{dQ}{dP}(X_i)$. Then by the SLLN it follows that $S_t/t\to \KL(Q\|P)$ almost surely under $Q$, and so the likelihood-ratio threshold tests will be pathwise optimal.

\section{Experimental Setup}\label{sec:exp1}
We validate our findings with a number of experiments. Our profile plots prove that $\log W(x)/x$ determines the asymptotic $e$-power, as opposed to the number of tests alone, for example. In our SPRT results, we confirm that the two items we need are size control under the null and first-passage scaling under the alternative (in our Gaussian setting). Our Monte Carlo $e$-power curves validate the resulting schemes of $2^{-k}$ counting (which is subexponential), power counting (which approaches $I/(1+\varepsilon)$), log-corrected and iterated-log unweighted schedules (which achieve the full rate $I$), and the weighted geometrically-spaced counting (which also recovers the full rate $I$). Further, we run null-side experiments to, in the finite sample case, verify $e$-process validity at deterministic times and grid-defined stopping times. Lastly, our delayed-test and weak-optimality experiments illustrate our theorem's weakness. For one, aggregation will transfer the speed of whatever based tests are used. Second, expectation optimality leads to random limiting growth rates instead of a deterministic $\rho I$. We summarize all the schemes tested in Table~\ref{tab:schemes}.

\begin{table}[h]
\centering
\caption{Primary schemes tested. The target column uses $I=1/2$.}
\label{tab:schemes}
\begin{tabularx}{\textwidth}{@{}lXcc@{}}
\toprule
Scheme & Level and weight definition & \(\rho\) & Target \(\rho I\) \\
\midrule
Dyadic count & \(\alpha_k=2^{-k}\), \(w_k=1\) & 0 & 0 \\
Power count, \(\varepsilon=0.5\) & \(\alpha_k=c_{\varepsilon}k^{-(1+\varepsilon)}\), \(w_k=1\) & 2/3 & 1/3 \\
Power count, \(\varepsilon=0.1\) & \(\alpha_k=c_{\varepsilon}k^{-(1+\varepsilon)}\), \(w_k=1\) & 10/11 & 0.4545 \\
Log-corrected count & \(\alpha_k=c/[k(\log k)^2]\), \(k\ge 10\), \(w_k=1\) & 1 & 0.5 \\
Iterated-log count & \(\alpha_k=c/[k\log k(\log\log k)^2]\), \(k\ge 16\), \(w_k=1\) & 1 & 0.5 \\
Weighted dyadic & \(\alpha_k=2^{-k}\), \(w_k=(6/\pi^2)2^k/k^2\) & 1 & 0.5 \\
Fractional weighted dyadic & \(\alpha_k=2^{-k}\), \(w_k\propto 2^{\rho k}/k^2\), \(\rho=0.5\) & 0.5 & 0.25 \\
\bottomrule
\end{tabularx}
\end{table}

We work in the simplest setting under which the base-level tests are pathwise optimal. In all the experiments, we use independent observations from either $P:X_i\sim N(0,1)$ or $Q:X_i\sim N(\mu, 1)$. In all our experimental results, we use $\mu=1$. Therefore, $I=\KL(Q\|P)=\mu^2/2=0.5$. The log-likelihood ratio process here is
\[
S_t=\sum_{i=1}^t \left(\mu X_i - \frac{\mu^2}{2}\right).
\]
This level $\alpha$ one sided SPRT will reject when
\[
\tau_{\alpha}=\inf\{t\ge 0: S_t> \log(1/\alpha)\}.
\]
Since Gaussian increments are continuous, in our simulations, the difference between equality and strict inequality is trivial. We evaluate the running maximum $H_t=\max_{0\le s\le t} S_s$ so that $\tau_{\alpha_k}\le t$ is equivalent to $H_t\ge b_k$. As such, every aggregate can be viewed as $M_t=W(H_t)$, allowing us to avoid simulating each of the stopping times separately.

In all log-power plots we use the strict positivity correction of
\[
\log M_t^{(\eta)}=\log\{\eta+(1-\eta)M_t\}.
\]
Here, $\eta=10^{-12}$. This will prevent $\log 0$ at early terms and is asymptotically negligible as $M_t\to\infty$. Now, for a particular Monte Carlo mean curve we compute the pointwise standard errors as
\[
\operatorname{SE}\{\widehat m(t)\}=\frac{s_t}{\sqrt{n}}.
\]
Here, $s_t$ is the sample standard deviation across independent Monte Carlo paths at time $t$. In addition, the $e$-power mean at time $t$ is defined as
\[
\epower(t)=\frac{1}{n}\sum_{i=1}^n \frac{\log\{\eta+(1-\eta)M_t^{(i)}\}}{t}.
\]
The $e$-power figures plot the means without shaded pointwise standard-error bars for brevity and to ensure that the curves remain legible. However, all standard errors are computed numerically and may be be printed from the saved arrays when our code is run. In the SPRT size experiment, all the plotted vertical bars are normal-approximation $95\%$ Monte Carlo intervals $\widehat p\pm1.96\sqrt{\widehat p(1-\widehat p)/n}$.

\section{Experimental Results}\label{sec:exp2}
\subsection*{Experiment 1: Weight Profiles}
In this experiment, we evaluate the deterministic profile $W(x)$ for each schedule. Our theorem~\ref{thm:main} predicts that $\log W(x)/x\to \rho$. The validity constraint entails that the envelope $\log W(x)-x\le 0$. In particular, this profile grid contains $700$ points over $[2,1650]$. We do not use any random simulation here. 

Our results in Figure~\ref{fig:combined-profiles} match what we would expect from the theory. The $2^{-k}$ counting has exponent tending to zero. The power-schedules approach $1/(1+\varepsilon)$. The log-corrected, iterated-log, and full weighted schedules approach an exponent of one. The envelope plot also confirms that each profile remains below the $e^x$ upper bound. This is important because under the alternative such experiments can only be interpreted correctly if the schedules behave as expected with regard to the predicted profiles.

\begin{figure}[h]
    \centering
    \begin{subfigure}[b]{0.48\textwidth}
        \centering
        \includegraphics[width=\textwidth]{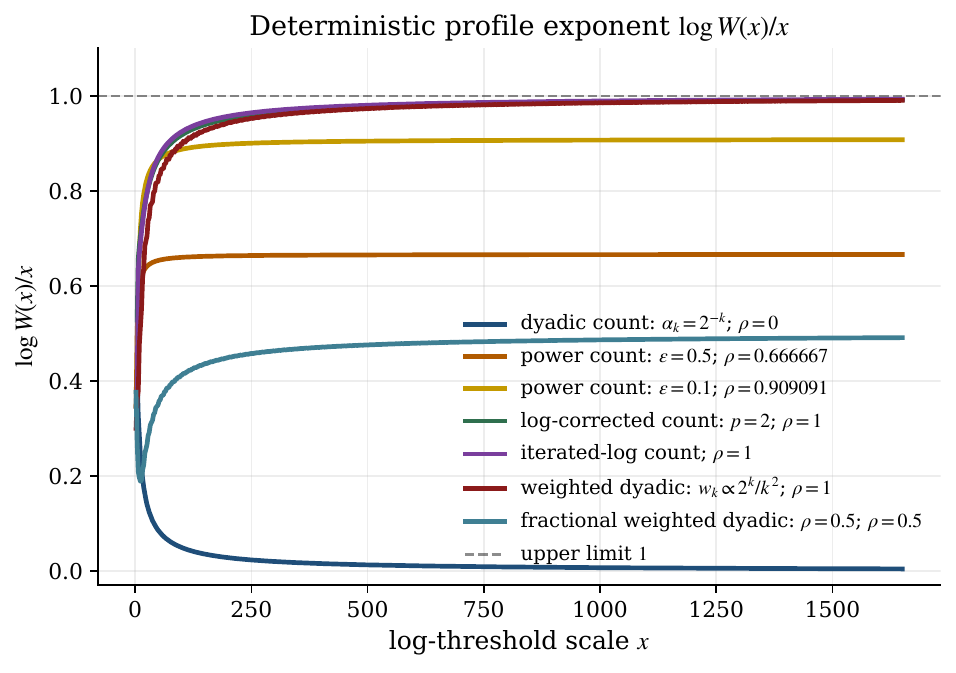}
        \caption{Deterministic profile exponents}
        \label{fig:profile-exponent}
    \end{subfigure}
    \hfill
    \begin{subfigure}[b]{0.48\textwidth}
        \centering
        \includegraphics[width=\textwidth]{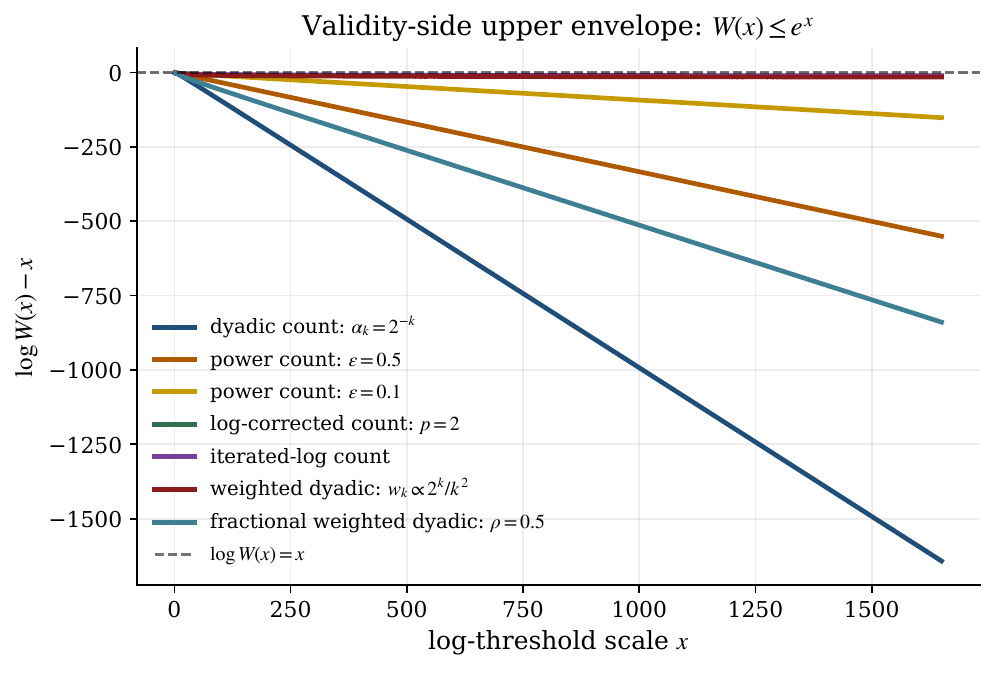}
        \caption{Validity-side envelope}
        \label{fig:profile-envelope}
    \end{subfigure}   
    \caption{Comparison of profile exponents (a) and validity-side envelopes (b). In (a), the deterministic profile exponents are $\log W(x)/x$, and the curves approach the predicted $\rho$ values that are listed in Table~\ref{tab:schemes}. In (b), this is the validity-side envelope $\log W(x)-x$. Staying below zero is the numerical analogue of $W(x)\le e^x$.}
    \label{fig:combined-profiles}
\end{figure}

\FloatBarrier

\subsection*{Experiment 2: Finite horizon SPRT size under the null}
In this experiment, we test that the one-side likelihood ratio SPRT obeys
\[
P\left(\sup_{t\le T} S_t\ge \log(1/\alpha)\right)\le\alpha.
\]
at every finite horizon. This must be true because the likelihood ratio process is a nonnegative martingale under $P$. In our experiment, we estimate this crossing probability over a grid of $\alpha$ values. We run using $80,000$ null paths, a horizon $T=1800$, and $24$ logarithmically spaced levels $\alpha\in[0.002,0.25]$. All error bars are $\widehat p\pm1.96\sqrt{\widehat p(1-\widehat p)/80000}$.

In Figure~\ref{fig:sprt-size}, all the empirical crossing probabilities lie below or at the diagonal up to Monte Carlo uncertainty. This confirms correctness of our size requirement in the aggregation theorem (for the Gaussian test family). Albeit this is a finite-horizon experiment and therefore weaker than the infinite-horizon theorem, it serves as a sanity check for all following experiments. Note that we do not optimize for being tight at the boundary $\alpha$, as it is not necessary for the purposes of this work. However, note that there are works that indeed do this \citep{fischer2026improving}.

\begin{figure}[h!]
\centering
\includegraphics[width=0.7\textwidth]{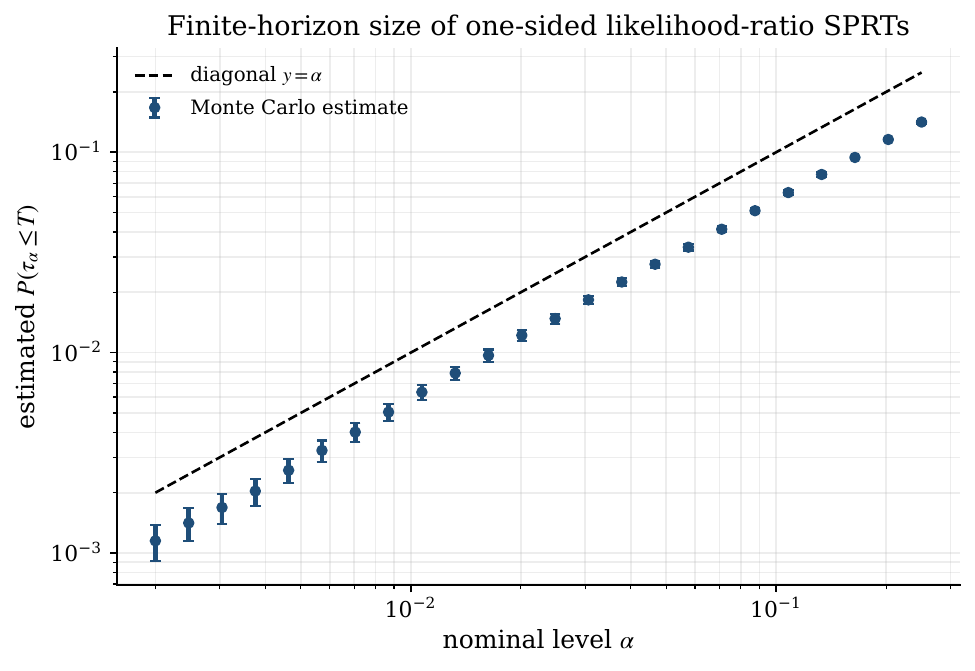}
\caption{Finite-horizon size check for the Gaussian one-sided SPRT. The diagonal is the nominal level $\alpha$.}
\label{fig:sprt-size}
\end{figure}

\FloatBarrier

\subsection*{Experiment 3: First-passage scaling under the alternative}
In this experiment, we verify that the first passage time for threshold $b$ should satisfy $\tau_b/b\to 1/I=2$. This must hold since under $Q$, the likelihood ratio score satisfies $S_t/t\to I$. And this is the base-test optimality condition that our main theorem requires. For the experiment, we run using $5000$ alternative paths, a horizon $T=4000$, and $70$ thresholds $b\in[8,1560]$. The plotted band is the empirical $10\%$-$90\%$ range of $\tau_b/b$. The mean and median are both shown separately.

In Figure~\ref{fig:hitting-ratio}, we see the mean and median crossing-time ratios move toward the target $2$. The quantile band shrinks on the scale as the threshold grows. Therefore, we validate using the Gaussian SPRT as a pathwise asymptotically optimal base family for the aggregation experiments.

\begin{figure}[h]
\centering
\includegraphics[width=0.7\textwidth]{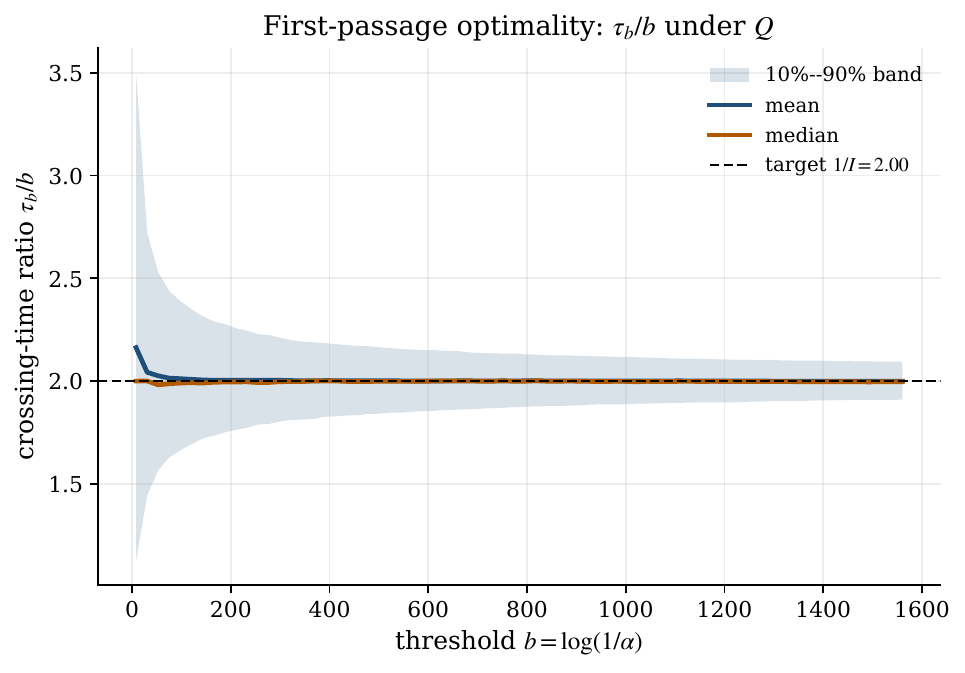}
\caption{First-passage ratio \(\tau_b/b\) under \(Q\). The target is \(1/I=2\).}
\label{fig:hitting-ratio}
\end{figure}

\FloatBarrier

\subsection*{Experiment 4: The $e$-power under the alternative}
The main experiment here estimates the $e$-power as
\[
\frac{1}{t}\EE_Q\left[\log\{\eta+(1-\eta)M_t\}\right]
\]
for each schedule. The expected outcome here is a hierarchy of limiting rates. That is, zero for the $2^{-k}$ counting, $I/(1+\varepsilon)$ for the power counting, $I$ for the full rate profiles, and $\rho I$ for the fractional weighted $2^{-k}$ profiles. We run $20000$ alternative paths up to $T=2500$. The time grid has $291$ unique points, consisting of an early geometric grid, a linear grid of $240$ points along with time zero. To reduce the Monte Carlo noise in between-schedule comparisons, we reuse the same simulated $H_t$ matrix across all schedules.

In Figure~\ref{fig:main-combined}, all the curves indeed follow the partial ordering. In particular, the $2^{-k}$ counting scheme grows much too quickly on the exponential scale. The $\varepsilon=0.5$ power schedule approaches a lower target than the $\varepsilon=0.1$ power schedule. All of the log-corrected, iterated-log, and weighted dyadic approach the full-rate line. And, the fractional weighted dyadic schedule closely follows the target $0.25$. In the log-growth plot, we present the same result but as a slope comparison instead.

\begin{figure}[h]
    \centering
    \begin{subfigure}[b]{0.48\textwidth}
        \centering
        \includegraphics[width=\textwidth]{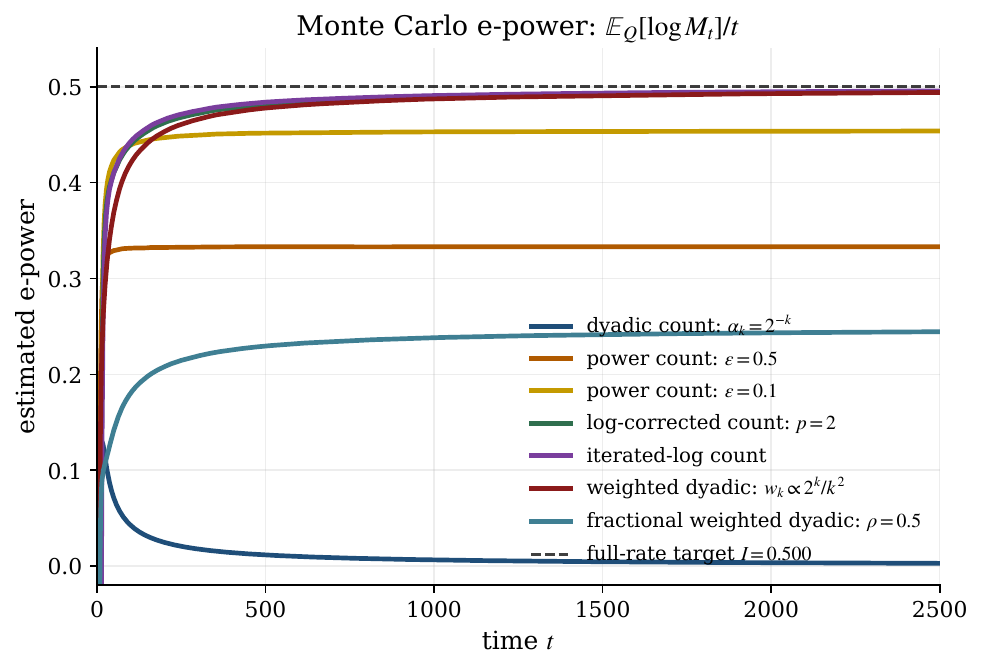}
        \caption{Monte Carlo e-power curves under \(Q\). Horizontal reference: full-rate target \(I=0.5\).}
        \label{fig:main-epower}
    \end{subfigure}
    \hfill
    \begin{subfigure}[b]{0.48\textwidth}
        \centering
        \includegraphics[width=\textwidth]{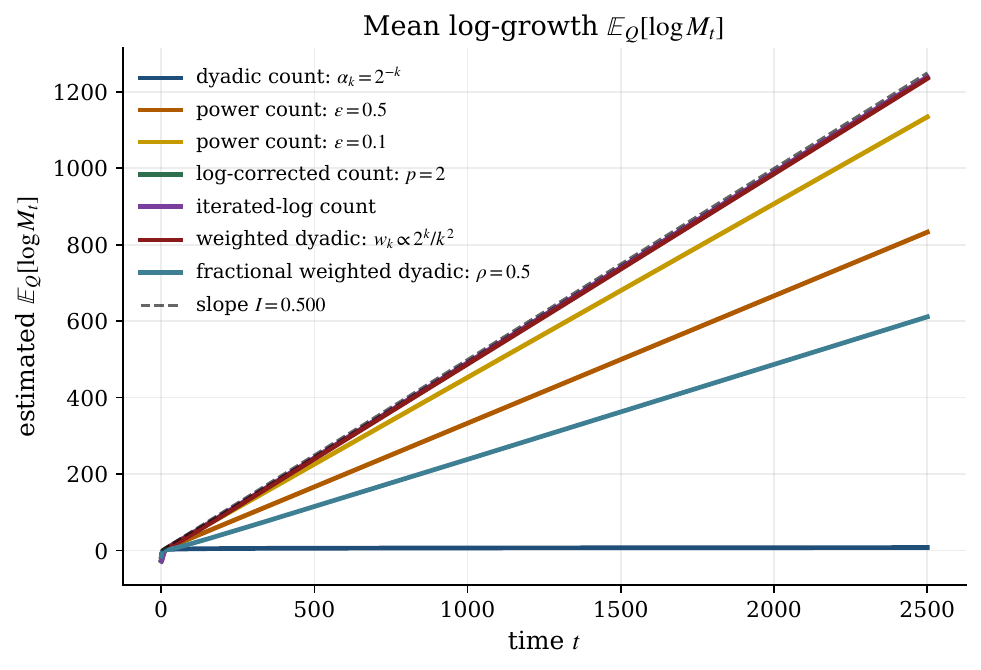}
        \caption{Estimated mean log-growth \(\EE_Q[\log M_t]\). A full-rate process should have asymptotic slope \(I=0.5\).}
        \label{fig:main-loggrowth}
    \end{subfigure}
    \caption{Analysis of growth metrics under \(Q\): e-power curves (a) and mean log-growth (b).}
    \label{fig:main-combined}
\end{figure}

\FloatBarrier

\subsection*{Experiment 5: Finite-time speed among the full-rate schemes}
In this experiment, we go a bit further by noting that the finite-time approach of the schemes to $I$ may vary significantly. This is because the schemes only share the asymptotic rate $I$. We compare the gap $I-\widehat \PP_{M}(t)$, the first grid times at which the full-rate scheme reaches $80\%$, $90\%$, and $95\%$ of $I$. In addition, we compare the final-time distribution of $t^{-1}\log M_t$. All these figures are directly derives from the same $20000$ alternative paths that are used in Experiment 4. We do not perform any further random simulation for this experiment.

Importantly, the results here in Figure~\ref{fig:full-rate-analysis} can be intuited as follows. The full-rate idea we present is asymptotic, but does not guarantee identical finite-time performance. Both the weighted dyadic and log-corrected schedules certainly may differ in their speed to reaching the $\KL$ time. Further, in the final-time boxplots, we show that there is residual finite-time dispersion even if the limiting rate is the same. Thus, here we give finite horizon results.

\begin{figure}[h]
    \centering
    \begin{subfigure}[b]{0.31\textwidth}
        \centering
        \includegraphics[width=\textwidth]{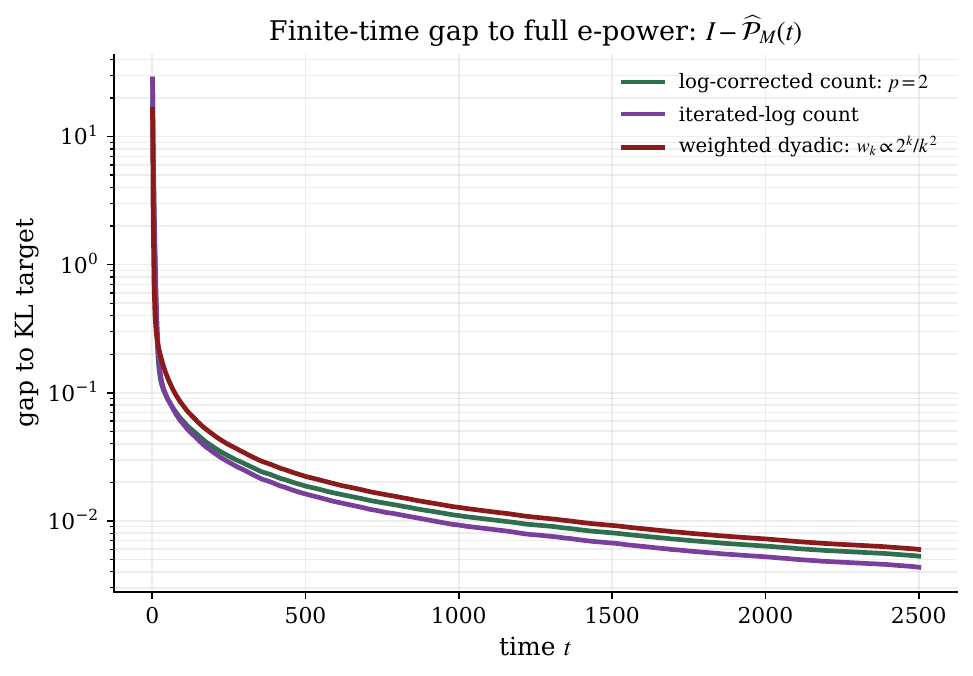}
        \caption{Finite-time gap to full e-power for the full-rate schedules. The vertical scale is logarithmic.}
        \label{fig:full-gap}
    \end{subfigure}
    \hfill
    \begin{subfigure}[b]{0.31\textwidth}
        \centering
        \includegraphics[width=\textwidth]{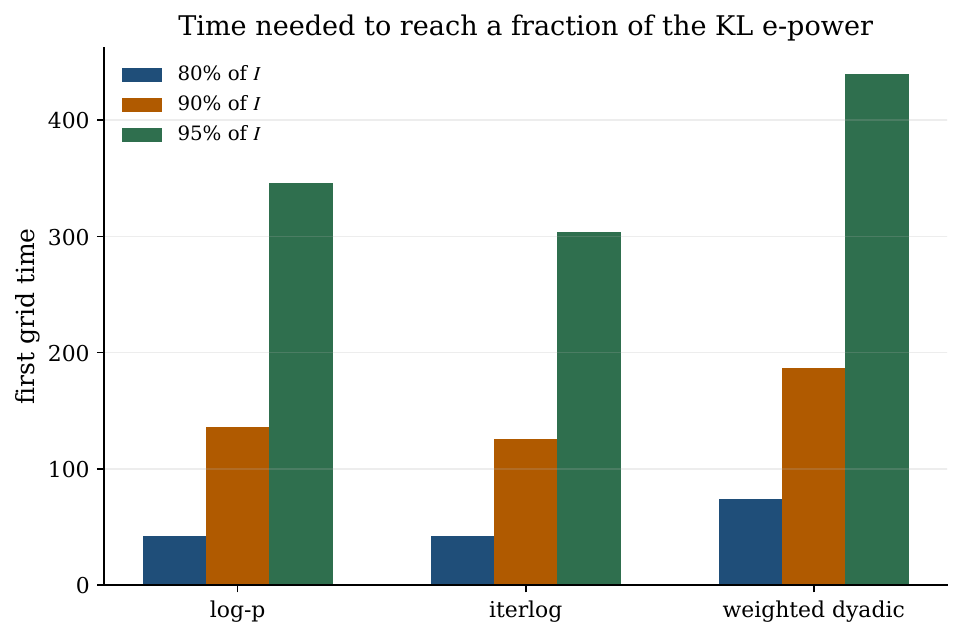}
        \caption{First simulated grid time at which each full-rate schedule reaches 80\%, 90\%, and 95\% of \(I\).}
        \label{fig:reaching-times}
    \end{subfigure}
    \hfill
    \begin{subfigure}[b]{0.31\textwidth}
        \centering
        \includegraphics[width=\textwidth]{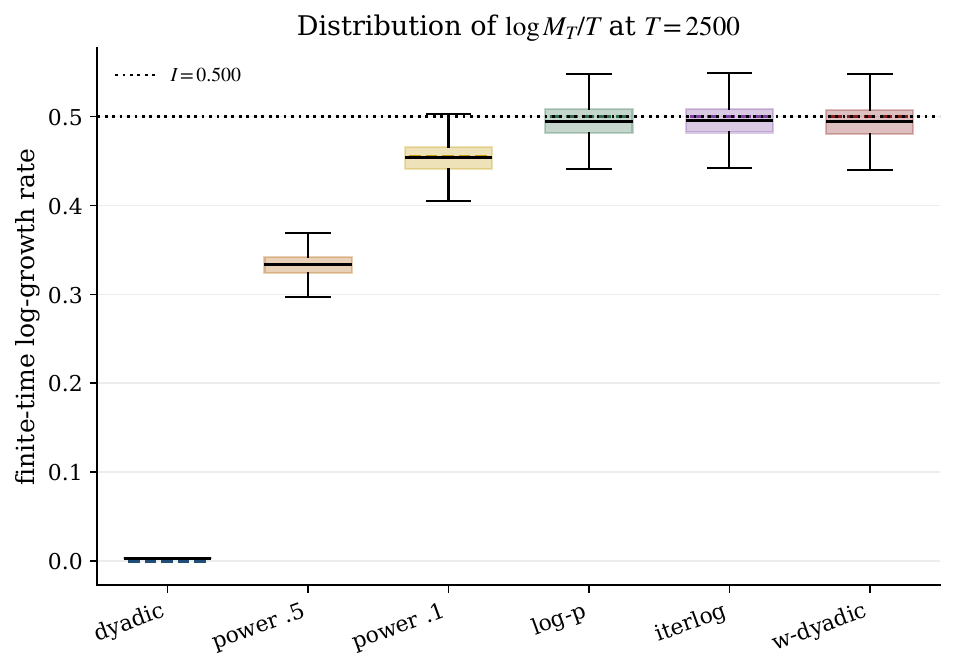}
        \caption{Distribution of final-time log-growth rates at \(T=2500\). Dashed within-box reference marks show each scheme's theoretical target.}
        \label{fig:final-dist}
    \end{subfigure}
    \caption{Performance analysis of full-rate schedules: e-power gaps (a), reaching times (b), and final growth distributions (c).}
    \label{fig:full-rate-analysis}
\end{figure}

\FloatBarrier

\subsection*{Experiment 6: Aggregate validity under the null}
In this experiment, we test that the $e$-process property holds. That is, for every finite stopping time $\sigma$, we must have that $\EE_P[M_{\sigma}]\le 1$. We check deterministic-time expectations and grid-defined stopping rules (e.g.\ stop the first time the grid-observed aggregate exceeds $c$; else stop at the final grid time). In the run, we use $50000$ null paths, a horizon $T=1800$ and a $170$ time point grid. We evaluate the dyadic count, power count with $\varepsilon=0.5$, log-corrected counts with $p=2$, and the full weighted dyadic count. We use $c\in\{1.5, 2, 5, 10\}$ for the stopping thresholds.

All of these expectations in Figure~\ref{fig:null-checks-combined} (i.e.\ stopped ones or deterministic-time ones) are consistent with the $e$-process boundary. Although they do not prove optional stopping, the simulations stress-test the numerical implementation of the null-side behavior, weighted sums, capital normalizations, and so on.

\begin{figure}[h]
    \centering
    \begin{subfigure}[b]{0.48\textwidth}
        \centering
        \includegraphics[width=\textwidth]{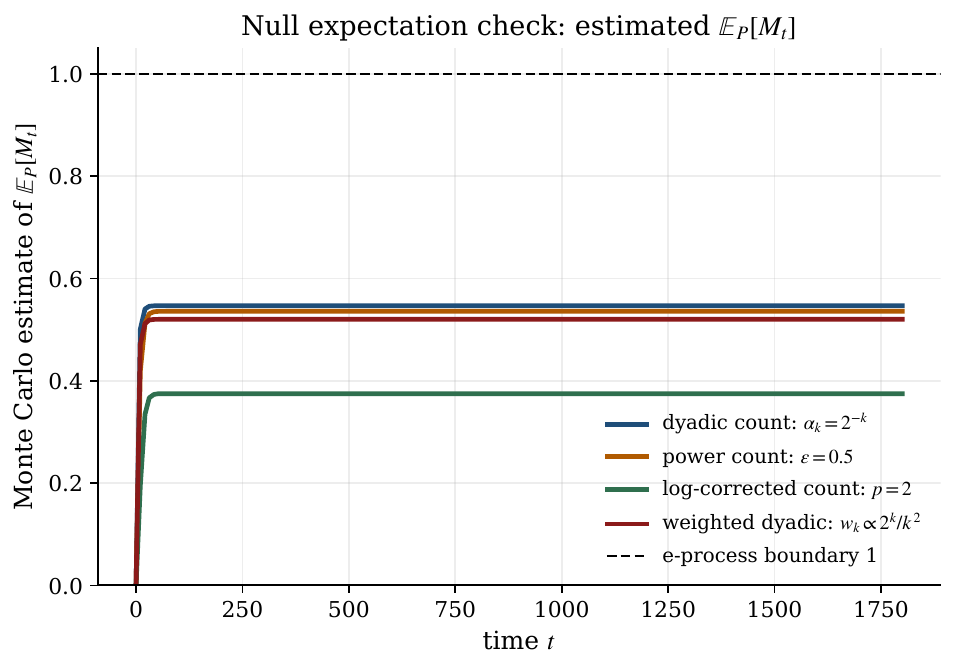}
        \caption{Estimated null expectations \(\EE_P[M_t]\) at deterministic times. The dashed line is the e-process boundary 1.}
        \label{fig:null-fixed}
    \end{subfigure}
    \hfill
    \begin{subfigure}[b]{0.48\textwidth}
        \centering
        \includegraphics[width=\textwidth]{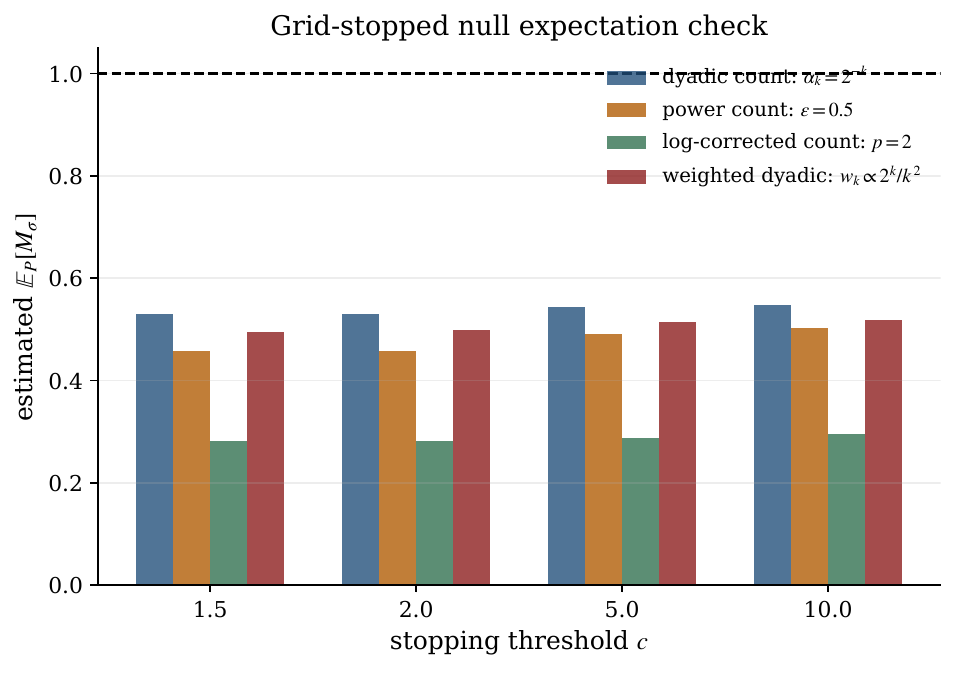}
        \caption{Estimated stopped null expectations for grid-defined stopping rules. Values remain at or below 1 up to Monte Carlo noise.}
        \label{fig:null-stopped}
    \end{subfigure}
    \caption{Validation of null properties: fixed-time expectations (a) and stopped expectations (b).}
    \label{fig:null-checks-combined}
\end{figure}

\FloatBarrier

\subsection*{Experiment 7: Robustness to different $\KL$ values}
In this experiment, we test if the weighted dyadic full-rate aggregate looks similar across different values of $\mu$ after normalizing the vertical axis by $I=\mu^2/2$ and plotting against the information time, i.e.\ $It$. We use an information horizon of $750$, $\mu\in\{0.5,0.8,1.0, 1.25\}$ with $N=6000$ paths per value. The corresponding time horizons to all these are $6000, 2344, 1500, 960$ respectively with $180$ grid points each.

As we can see in Figure~\ref{fig:kl-scaling}, all the normalized curves move toward the unit line on a common information-time scale. This corroborates the fact that $I$ determines the asymptotic growth (once the profile exponent is held constant at $\rho=1$).

\begin{figure}[h]
\centering
\includegraphics[width=0.7\textwidth]{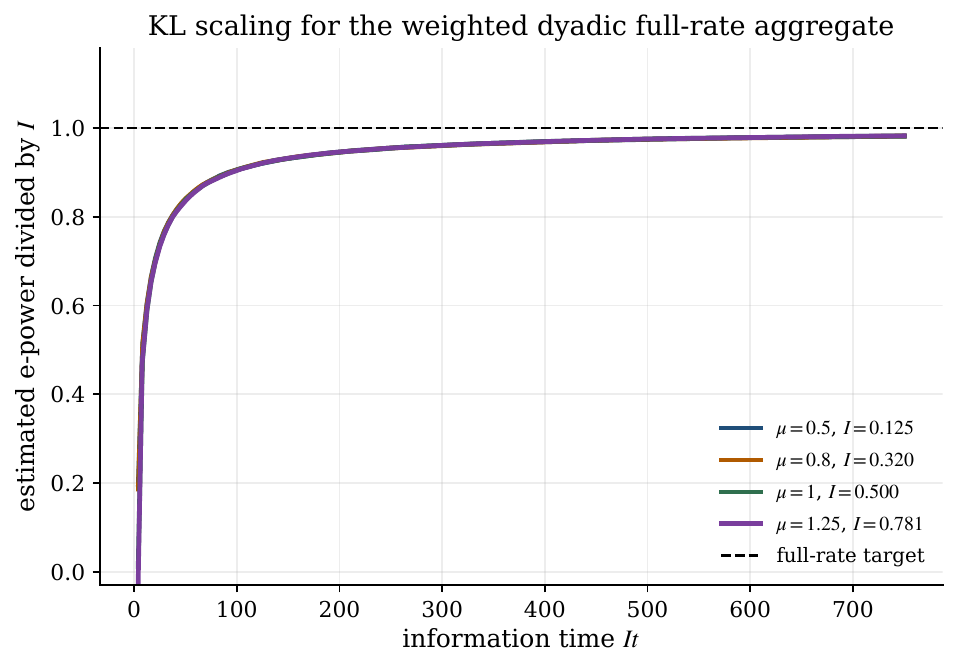}
\caption{KL scaling for the weighted dyadic full-rate aggregate. The vertical axis is e-power divided by \(I\); the horizontal axis is information time \(It\).}
\label{fig:kl-scaling}
\end{figure}

\subsection*{Experiment 8: What happens if we deliberately delay the tests?}
Our theorem assumes that the base tests are asymptotically optimal. However, if each rejection is delayed by some factor $\gamma> 1$, we preserve validity while losing alternative speed. Thus, the expected rate for a full-profile will become $I/\gamma$. For this, we use $10000$ alternative paths, a time grid through $T=2500$, and delays with $\gamma\in\{1,1.25,1.5,2\}$. We evaluate the weighted dyadic profile at $H_{\lfloor t/\gamma\rfloor}$.

In this Figure~\ref{fig:delayed}, we see that the delayed curves align with the predicted horizontal levels $I/\gamma$. Importantly, this shows us that while aggregation can optimize the allocation of levels and weights, it cannot ``repair'' base tests that are slow (i.e.\ suboptimal).

\begin{figure}[h]
\centering
\includegraphics[width=0.7\textwidth]{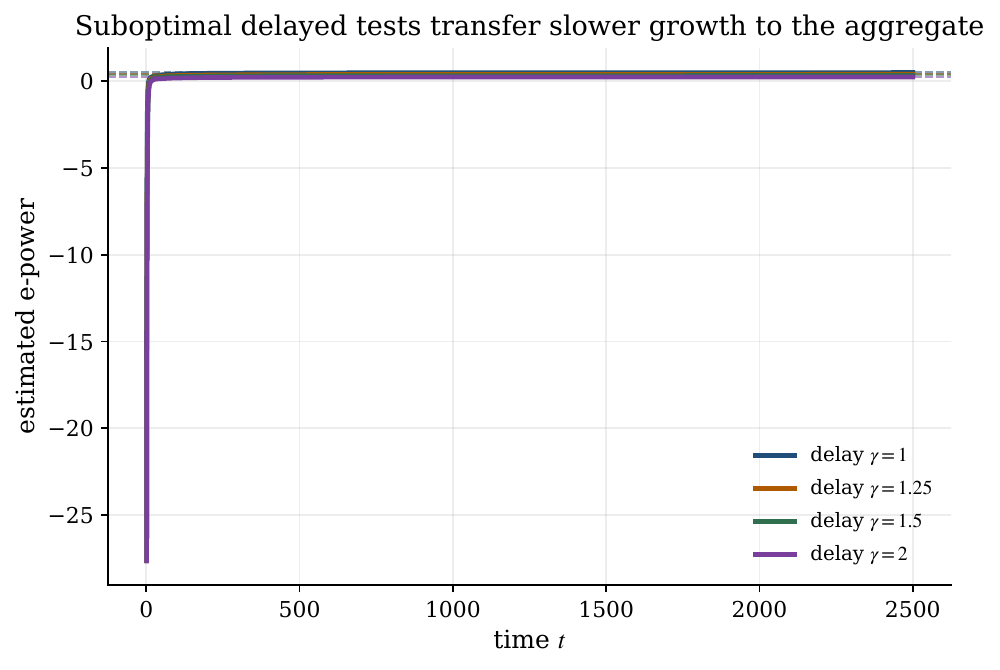}
\caption{Delayed tests transfer slower base-test speed into slower aggregate $e$-power. Dashed lines indicate \(I/\gamma\).}
\label{fig:delayed}
\end{figure}

\FloatBarrier

\subsection*{Experiment 9: Empirically validating that expectation optimality is not enough}
In Example~\ref{ex:expectation-too-weak}, we use a random multiplier $Y\in\{1/2,3/2\}$, each with probability $1/2$ and set $\tau_k\approx Yb_k/I$. In this case, since $\EE[Y]=1$, the hitting times are expectation-optimal (crudely). However, the pathwise growth becomes $M_t=W(It/Y)$ and
\[
\frac1t \log M_t\longrightarrow \frac{\rho I}{Y}.
\]
In the full-rate case, the limiting branches (as opposed to $I=1/2$ only) are $2I=1$ and $(2/3)I=1/3$. Our branch plot is deterministic; that is, the histogram samples 30000 independent values of $Y$ and then evaluates the final-time rates at $T=2500$.

In this two-branch behavior as in Figure~\ref{fig:weak-optimality-combined}, we empirically confirm that expectation (weak) optimality does not give the tail control needed for a deterministic growth rate of the aggregate. One would need something akin to pathwise optimality or nested-in probability control to avoid such situations.

\begin{figure}[h]
    \centering
    \begin{subfigure}[b]{0.48\textwidth}
        \centering
        \includegraphics[width=\textwidth]{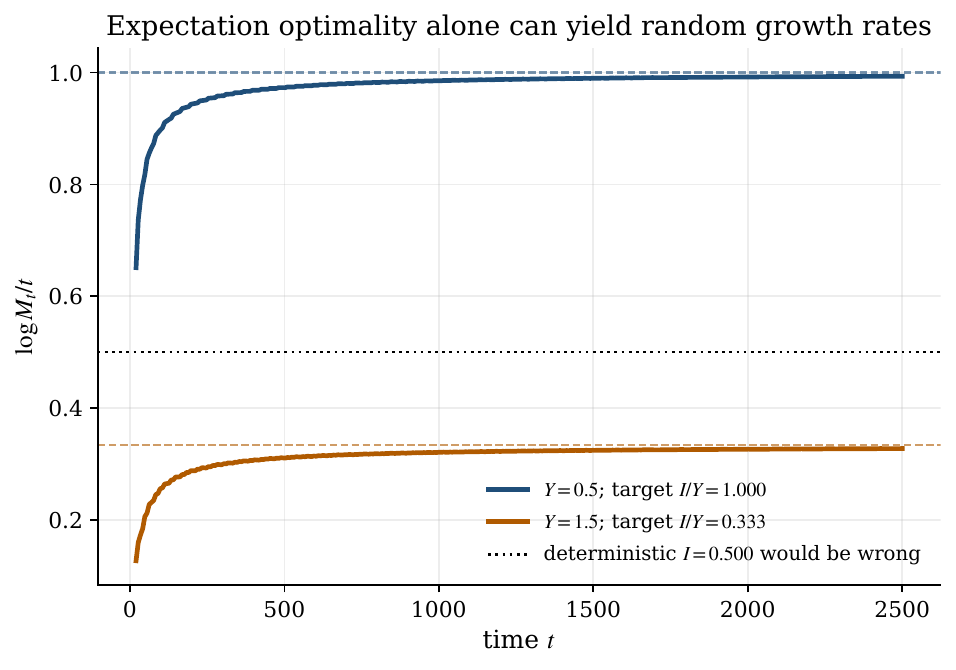}
        \caption{Two deterministic branches in the weak-optimality counterexample. The deterministic line \(I=0.5\) is not the limiting rate on either branch.}
        \label{fig:weak-branches}
    \end{subfigure}
    \hfill
    \begin{subfigure}[b]{0.48\textwidth}
        \centering
        \includegraphics[width=\textwidth]{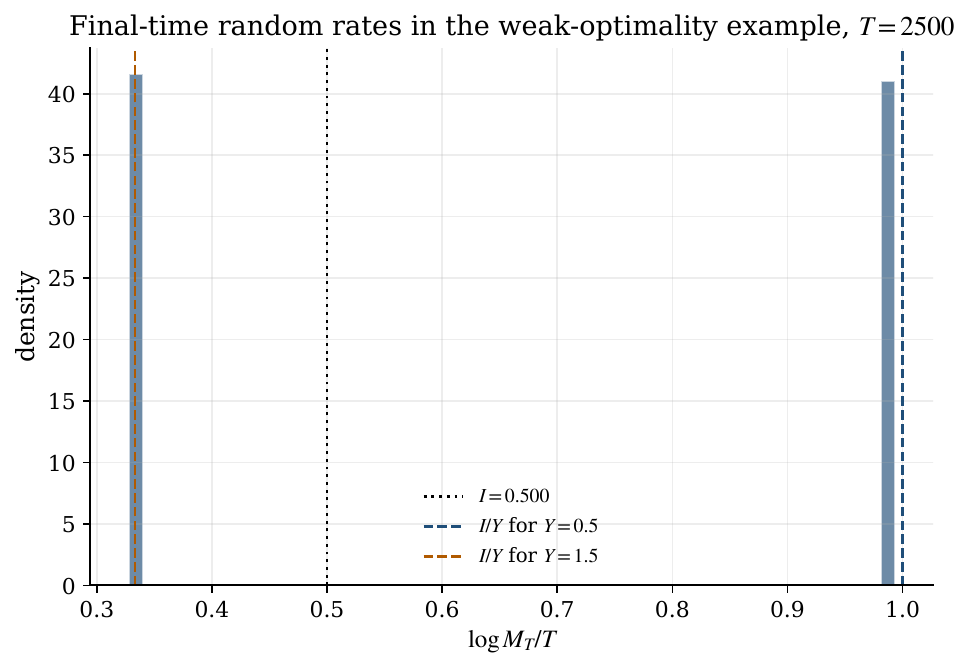}
        \caption{Final-time distribution of \(\log M_T/T\) in our weak-optimality counterexample. The separated modes reflect the random multiplier \(Y\).}
        \label{fig:weak-hist}
    \end{subfigure}
    \caption{Analysis of the weak-optimality counterexample: deterministic branches (a) and final-time distribution (b).}
    \label{fig:weak-optimality-combined}
\end{figure}

\FloatBarrier

\subsection*{Experiment 10: Capital-Budget and truncation analysis}
In this experiment, we recall that validity comes from the capital constraint $\sum_{k}w_k\alpha_k\le 1$. Thus, we plot partial budgets $\sum_{k\le K}w_k\alpha_k$ for several representative schedules. In doing so, we show the speed at which each schedule subsumes the budget. Specifically, our deterministic truncation index ranges geometrically up to $K=200000$. Both the dyadic and weighted dyadic budgets are exact. However, for the log-corrected and iterated-log schedules, we use finite sums and integral-tail upper bounds for choosing normalization constants (that are conservative).

As we can see, all the curves remain below the unit budget. While the logarithmic schedules depend on long tails, the dyadic ones take up capital very quickly. Thus, this result can be very useful for implementations as each finite code path will truncate or approximate an infinite schedule.

\FloatBarrier

\begin{figure}[h]
\centering
\includegraphics[width=0.4\textwidth]{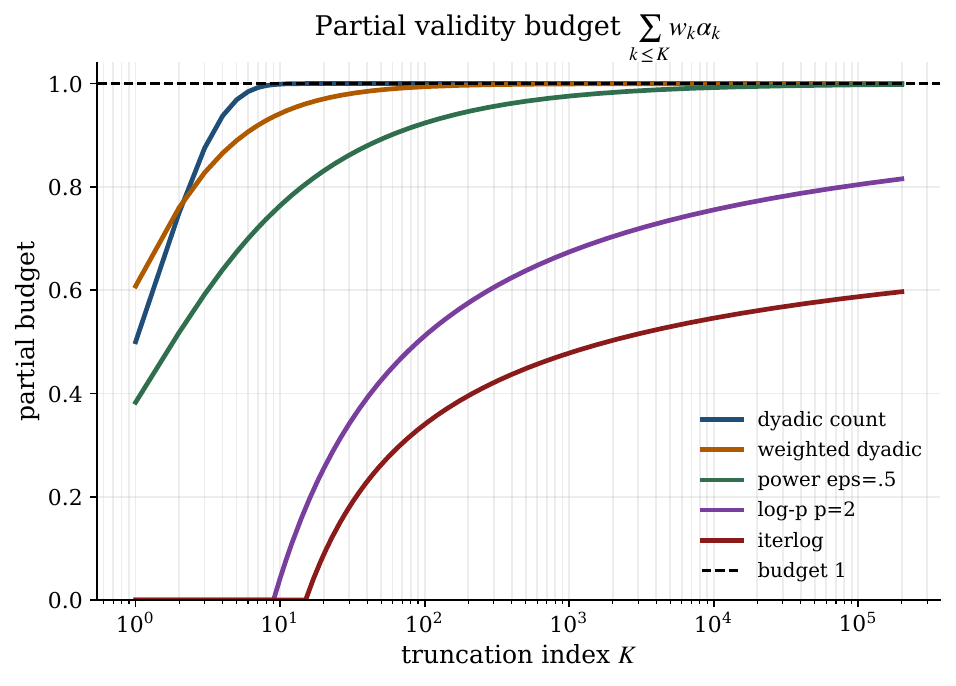}
\caption{Partial validity budgets \(\sum_{k\le K}w_k\alpha_k\). All curves stay at or below 1.}
\label{fig:capital-budget}
\end{figure}

\FloatBarrier

\end{document}